\documentclass[11pt]{article}

\usepackage[margin=1in]{geometry}

\usepackage{amsmath,amssymb,amsfonts}%
\usepackage{amsthm}%
\usepackage{mathrsfs}%
\usepackage{bbm}
\usepackage{mathtools}
\usepackage{MnSymbol}
\usepackage{dsfont}
\usepackage{mathrsfs}
\usepackage{xcolor}

\usepackage{natbib}
\setcitestyle{authoryear,open={(},close={)}}

\usepackage{bm}

\usepackage[hyperfootnotes=false]{hyperref}
\hypersetup{colorlinks,citecolor=blue,linkcolor=blue}

\theoremstyle{plain}

\newtheorem{theorem}{Theorem}
\newtheorem{lemma}[theorem]{Lemma}

\newtheorem{proposition}[theorem]{Proposition}

\theoremstyle{definition}
\newtheorem{definition}{Definition}
\newtheorem{example}{Example}
\newtheorem{remark}{Remark}

\def\eqd{\stackrel{\mbox{\scriptsize{d}}}{=}}
\def\simiid{\stackrel{\mbox{\scriptsize{\rm iid}}}{\sim}}

\newcommand{\E}{\mathbb{E}}

\newcommand{\W}{\mathcal{W}}
\newcommand{\R}{\mathbb{R}}
\newcommand{\X}{\mathbb{X}}
\newcommand{\Y}{\mathbb{Y}}

\newcommand{\cov}{\textup{Cov}}

\usepackage{xcolor}
\definecolor{lightorange}{rgb}{1.0, 0.65, 0.0}
\definecolor{lightbrown}{rgb}{0.800, 0.400, 0.200}
\usepackage{pgfplots}
\pgfplotsset{compat=1.9}

\title{Measures of Dependence based on Wasserstein distances}
\author{Marta Catalano\footnote{Luiss University, Italy. Email: \texttt{mcatalano@luiss.it}}~~and Hugo Lavenant\footnote{Bocconi University, Italy. Email: \texttt{hugo.lavenant@unibocconi.it}}}

\date{\today}

\begin{document}

\maketitle

\begin{abstract}
Measuring dependence between random variables is a fundamental problem in Statistics, with applications across diverse fields. While classical measures such as Pearson’s correlation have been widely used for over a century, they have notable limitations, particularly in capturing nonlinear relationships and extending to general metric spaces. In recent years, the theory of Optimal Transport and Wasserstein distances has provided new tools to define measures of dependence that generalize beyond Euclidean settings. This survey explores recent proposals, outlining two main approaches: one based on the distance between the joint distribution and the product of marginals, and another leveraging conditional distributions. We discuss key properties, including characterization of independence, normalization, invariances, robustness, sample, and computational complexity. Additionally, we propose an alternative perspective that measures deviation from maximal dependence rather than independence, leading to new insights and potential extensions. Our work highlights recent advances in the field and suggests directions for further research in the measurement of dependence using Optimal Transport.

\medskip

\end{abstract}

\section{Introduction}

Let $X,Y$ be two random variables defined on the same probability space. A measure of dependence is a summary of the joint law $\mathscr{L}(X,Y)$ with a real number $D(X,Y)$ such that $D(X,Y)=0$ if $X$ and $Y$ are independent. Beyond this, there is no universal agreement on the properties that $D(X,Y)$ should satisfy. The \emph{desiderata} have evolved over time and vary depending on the context. Nevertheless, the fundamental role of dependence in Statistics 
and in the analysis of real-world interactions has made it an active topic of study for almost two centuries. In recent years, the field has seen a resurgence of interest driven by the new computational and analytical tools that were previously infeasible. 

The first and still most used measure of dependence is Pearson's linear correlation, defined and studied in the 19th century by the joint efforts of many scholars including Bravais, Galton, Edgeworth and Pearson; see \cite{Pearson1920} for a personal historic note.  The reasons for its success can be ascribed to the intuitive interpretation, the easy computations, and the neat asymptotic theory. However, some of the most relevant properties, such as the detection of independence, are confined to jointly Gaussian random variables, it does not detect nonlinear relations, and it is only defined for random variables on $\mathbb{R}$. During the twentieth century, most of the literature has focused on defining measures of dependence for real-valued random variables with appealing properties that are not shared by Pearson's correlation, such as Spearman's $\rho$ \citep{Spearman1904}, Kendall's $\tau$ \citep{Kendall1938}, and many others. Nowadays, a major focus is on measuring the dependence of random vectors on $\mathbb{R}^m$ or random elements in more abstract metric spaces. This partly explains the recent success of Optimal Transport and Wasserstein distances in measuring dependence, due to their flexibility in defining a distance between probabilities on any metric space. The focus of this survey is to describe how they have been used and investigated in the recent literature. 

However, it should be noted that the first use of the Wasserstein distance to measure dependence between random variables on $\R$ can be traced back to Corrado Gini at the beginning of the 20th century; we refer to \cite{CifarelliRegazzini2017} for a historical account.
The Wasserstein distance on $\R$, equipped with the standard Euclidean distance, is intimately connected to Pearson's correlation. Indeed, this distance between two laws $\mathscr{L}(X)$ and $\mathscr{L}(Y)$ is based on the coupling $\mathscr{L}(X,Y)$ that maximizes Pearson's correlation, a.k.a. the optimal coupling or comonotonic pair. This link is \emph{not} useful to explain how the Wasserstein distance is now used to generalize a notion of correlation to arbitrary metric spaces, as we will describe in this work, but provides some intuition on why the Wasserstein distance on $\R$ had already appeared in Gini's work \citep{Gini1914} long before its modern development.

In the context of this work we will be using the Wasserstein distance because of its metric properties on the space of probability measures. A desirable feature for a measure of dependence is characterizing independence, in the sense that $D(X,Y)=0$ if and only if $X$ and $Y$ are independent. We will describe two main ways of using the Wasserstein distance to ensure this in a natural way: (i) evaluate the Wasserstein distance between $\mathscr{L}(X,Y)$ and the independence coupling $\mathscr{L}(X) \otimes \mathscr{L}(Y)$; (ii) evaluate the average Wasserstein distance between $\mathscr{L}(Y)$ and $\mathscr{L}(Y|X)$. Several proposals in the literature fall into these settings, starting from an early proposal by Gini \citep{CifarelliRegazzini2017}, to more recent works \citep{Ozair2019, mori2020earth, wiesel2022measuring, mordant2022measuring, borgonovo2024global, nies2025transport}. It should be noted that at this level of generality, the Wasserstein distance could be substituted by other non-degenerate discrepancy measures between probabilities on metric spaces, provided that one is able to find compelling alternatives.
For example, for random variables on $\R^m$ there are proposals involving mutual information \citep{Shannon1948} and a weighted $L_2$ distance between characteristic functions \citep{Szekely2007}, 
for random objects on more abstract metric and topological spaces the maximum mean discrepancy is a popular alternative (see, e.g., \cite{Gretton2005, deb2020measuring}. Interestingly, the seemingly unrelated proposal by \cite{Lyons2013} can also be reframed in terms of this metric \citep{Sejdinovic2013}.

As no single measure of dependence is universally superior to the others,
the choice of a specific metric must be guided by the theoretical properties required for the task at hand. Even though the general idea can be formulated with a generic distance,
several properties of the induced measure of dependence vary according to the specific metric. These include the notion of maximal dependence, which can be defined as the coupling that maximizes the distance from independence and naturally allows one to normalize the measure into an index of dependence in $[0,1]$, as well as invariance under marginal transformations, robustness properties, sample, and computational complexity. In this work we highlight many of these properties, reporting how the recent literature relying on Wasserstein distances has investigated them. We note that most of these properties can and have been investigated for other distance-based approaches, such as those mentioned above, and for other popular measures of dependence that are not based on distances (e.g., \citet{dette2013copula, chatterjee2021new}). 

The theoretical properties of Wasserstein-based dependence measures -- specifically their ability to capture non -linear associations and characterize independence -- have led to their adoption across a diverse range of statistical and machine learning tasks. A primary application is nonparametric independence testing, where it provides a robust test statistic against general alternatives \citep{warren2021wasserstein, wiesel2022measuring, nies2025transport}.
In the realm of model interpretability, these measures have been adapted for global sensitivity analysis, offering a framework to rank the importance of inputs for multivariate simulators \citep{borgonovo2024global}.
In the context of deep learning, optimal transport is increasingly used for representation learning \citep{Ozair2019, XiaoWang2019, eikenberry2025invariant}.
Furthermore, to mitigate the computational burden of optimal transport in high-dimensional settings, applications often rely on entropic regularizations \citep{Liu2022} or on the sliced Wasserstein distance. This variant has enabled scalable feature screening in high-dimensional data \citep{li2023scalable} and has been applied to cross-domain recommendation systems to align distributions across different domains \citep{zhao2021sliced}. Finally, the framework has found utility in finance, specifically for clustering financial time series by measuring the distance between empirical copulas \citep{Marti2017}.


In several settings the focus is not on detecting independence but rather some form of functional relation between $X$ and $Y$, which intuitively corresponds to a notion of maximal dependence.
For example, especially if $X$ and $Y$ have the same marginal distribution, a natural notion of maximal dependence is $X = Y$ almost surely, but several alternatives have been discussed as we explain in this survey. We argue that another and less explored use of the Wasserstein distance could be to measure the distance from maximal dependence rather than independence. This is the approach that has been used in \cite{Catalano2021,Catalano2024} to measure dependence between completely random measures and L\'evy measures in the context of Bayesian nonparametric Statistics. We show how some of the results therein can be adapted to real-valued random variables as well, with compelling benefits from both the point of analytical evaluations and numerical computations.
In particular, under some symmetry assumptions and for fixed marginal distributions, we provide a new result (Theorem~\ref{th:max}) that finds the maximal Wasserstein distance from maximal dependence and thus enables the definition of a new index of (maximal) dependence. This alternative approach of using the Wasserstein distance from complete dependence follows similar desiderata to the traditional measures of \emph{concordance} between real-valued random variables \citep{Scarsini1984}, 
which allows one to distinguish positive from negative dependence. We wonder if the new computational and analytical tools from Optimal Transport could be used to extend these concepts to random vectors and more abstract metric spaces as well. \\

\emph{Structure of the work.} After recalling the definition of Wasserstein distance in Section~\ref{sec:wass}, in Section~\ref{sec:measures} we describe two different ways to use the Wasserstein distance to build a measure of dependence that characterizes independence. In Section~\ref{sec:index} we discuss different normalizations to obtain an index of dependence in $[0,1]$, whose analytic expressions are discussed in Section~\ref{sec:analytic}. In Section~\ref{sec:invariances} we investigate some invariances of the indices with respect to marginal transformations and relate them to the notion of maximal dependence. In Sections~\ref{sec:robustness}, \ref{sec:sample_comp}, \ref{sec:comp_comp} we discuss some important properties that have recently emerged in the literature: robustness, sample, and computational complexity.
In Section~\ref{sec:alternative} we describe an alternative strategy, where the role of independence as baseline coupling is substituted by maximal dependence. In Section~\ref{sec:further} we mention some interesting frameworks that have used the Wasserstein distance with flavors different from those we have described in this work. Conclusions are presented in Section~\ref{sec:conclusions}. \\

\emph{Notations.} 
$\mathscr{L}(X)$ denotes the law of the random variable $X$. $\mathcal{P}(\X)$ denotes the set of probability distributions on a space $\X$. We always consider a probability space $(\Omega,\mathcal{A},\mathbb{P})$ and $(X,Y) : \Omega \to \X \times \Y $ a pair of random variables valued in $\X$ and $\Y$ Polish spaces.

\section{Wasserstein distances}
\label{sec:wass}

Optimal transport is a rich and well-established theory that studies optimal ways to couple random variables and defines distances between their probability distributions, a.k.a. Wasserstein distances. We refer to \citet{Santambrogio2015,villani2009optimal} for general introductions to the topic, and to \citet{panaretos2020invitation,chewi2024statistical} for overviews of its use in Statistics.
%
%

In this survey we mostly use Wasserstein distances between probability distributions on Polish spaces $\mathbb{X} \times \mathbb{Y}$ or $\mathbb{Y}$. With these examples in mind, consider a metric space $(\mathbb{M},d)$ and let $P,Q$ be two probability distributions on $\mathbb{M}$. We denote by $\Gamma(P,Q)$ the set of couplings between $P$ and $Q$, that is, the set of joint laws $\mathscr{L}(Z,Z')$ on $\mathbb{M} \times \mathbb{M}$  with fixed marginals $\mathscr{L}(Z) = P$ and $\mathscr{L}(Z') = Q$.  
By a slight abuse of notation, we sometimes use $\Gamma(Z,Z')$ as $\Gamma(\mathscr{L}(Z),\mathscr{L}(Z'))$.
For any $p \geq 1$, we define the Wasserstein distance with ground metric $d$ and order $p$ as 
\begin{equation}
\label{def:wass}
\W_{d,p}(P,Q) = \left( \min_{\gamma \in \Gamma(P,Q)} \E_{(Z,Z') \sim \gamma}(d(Z,Z')^p) \right)^{1/p}.  
\end{equation}
When $d$ and $p$ are clear from the context, we denote it simply by $\W$. The minimum is always attained  \citep[Theorem 4.1]{villani2009optimal} and a coupling $\gamma^*$ that realizes it
is called an optimal coupling. When $\gamma^* = \mathscr{L}(Z, T(Z))$ with $T : \mathbb{X} \to \mathbb{Y}$, the latter is termed an optimal transport map. 

We emphasize that $\W_{d,p}(\mathscr{L}(Z),\mathscr{L}(Z'))$ defines a distance between laws of random variables with finite $p$-th moments and thus takes in account the whole distribution of $Z$ and $Z'$, and not only a finite number of moments.
In the next sections we will discuss many properties of this distance that are useful to the measurement of dependence: analytic expressions in Section~\ref{sec:analytic}, topology induced by $\W$ in Section~\ref{sec:robustness}, sample and computational complexities in Section~\ref{sec:sample_comp} and \ref{sec:comp_comp} respectively.

\section{Measuring the proximity to independence}
\label{sec:measures}

An intuitive way to measure dependence between two random variables $(X,Y)$ is to measure their proximity to independence. There are two common characterizations of independence, one that is based on the joint law $\mathscr{L}(X,Y)$ and one that is based on the conditional law $\mathscr{L}(Y|X)$. Each characterization determines a natural way to use the Wasserstein distance to measure dependence.

Two random variables $(X,Y): \Omega \to \X \times \Y$ are independent if, equivalently,
\begin{enumerate}
\item $\mathscr{L}(X,Y) = \mathscr{L}(X)  \otimes \mathscr{L}(Y)$, 
\item $\mathscr{L}(Y|X) = \mathscr{L}(Y) \quad \mathscr{L}(X)\rm{-a.s.}$,
\end{enumerate}
where $\otimes$ denotes the product measure. From (i) one derives a \emph{joint} measure of dependence that evaluates the Wasserstein distance from the independent coupling on $\mathbb{X} \times \mathbb{Y}$, from (ii) a \emph{conditional} measure of dependence that evaluates the average Wasserstein distance between the conditional and the marginal distribution on $\mathbb{Y}$. This leads to, respectively,
\begin{equation} \label{defs}
\begin{aligned}
&D_{ \otimes}(X,Y) = \mathcal{W}(\mathscr{L}(X,Y),  \mathscr{L}(X)  \otimes \mathscr{L}(Y)), \\
&D_{|\bullet}(X,Y) =  \E_X(\mathcal{W}(\mathscr{L}(Y|X), \mathscr{L}(Y))^p) ^{1/p},
\end{aligned}
\end{equation}
where $\mathcal{W} = \mathcal{W}_{d,p}$ is a Wasserstein distance. Note that $D_\otimes$ requires a Wasserstein distance on $\mathcal{P}(\mathbb{X} \times \mathbb{Y})$ (and thus a distance $d$ on $\mathbb{X} \times \mathbb{Y}$), while $D_{|\bullet}$ requires a Wasserstein distance over the smaller space $\mathcal{P}(\mathbb{Y})$ (and thus a distance $d$ only over $\mathbb{Y}$).
This implies that the ground metric on $\mathbb{X}$ does not play a role in $D_{|\bullet}$, which can be seen as an advantage or 
not depending on our knowledge of $\mathbb{X}$.

Many works in the literature fall into the proposals in \eqref{defs}: we refer to \citet{Ozair2019, mori2020earth, mordant2022measuring,nies2025transport} for $D_\otimes$ and to Corrado Gini \citep{CifarelliRegazzini2017}; \citet{wiesel2022measuring,borgonovo2024global} for $D_{|\bullet}$. In the following, we will denote by $D$ any measure of dependence of the type $D_{ \otimes}$ or $D_{|\bullet}$. When we need to underline the dependence on $d,p$ we write $D(X,Y; d,p)$. As Wasserstein distances are distances, thus non-degenerate, we obtain the following.

\begin{lemma}
\label{th:independence}
Let $(X,Y):\Omega \to \mathbb{X} \times \mathbb{Y}$ random variables. Then, the following conditions are equivalent:
\begin{itemize}
\item[(i)] $(X,Y)$ are independent; 
\item[(ii)] $D_{ \otimes}(X,Y) = 0$; 
\item[(iii)] $D_{|\bullet}(X,Y) = 0$.
\end{itemize}
\end{lemma}

This property is often referred to as the \emph{characterization or detection of independence}, and it is strictly linked to the fact that $D_{ \otimes}$ and $D_{|\bullet}$ take into account the whole distribution of $(X,Y)$. Indeed, it typically does not hold for measures based on only the first two moments of the distribution, such as Pearson's correlation. 

A key point of the constructions above is that $\mathbb{X}$ and $\mathbb{Y}$ are not restricted to $\R$ or $\R^m$, but could in principle be defined on any Polish space. A major difference is that $D_{ \otimes}$ is symmetric in $(X,Y)$ if the distance $d$ is symmetric (see Remark~\ref{rem:symmetry} below), while $D_{|\bullet}$ is not:
by inverting the role of $X$ and $Y$ in $D_{|\bullet}$ we obtain a different measure, requiring only a distance on $d_\mathbb{X}$, which also characterizes independence.

\begin{remark}
\label{rem:symmetry}
For any distances $d_\mathbb{X}$, $d_\mathbb{Y}$ on $\X$, $\Y$ respectively, there is more than one way to define a distance on $\X \times \Y$.  The most common ones correspond to $\ell_q$ norms for any $q \geq 1$,
\[
d_{\mathbb{X} \times \mathbb{Y}}((x,y), (x',y')) = (d_\mathbb{X}(x,x')^{q} + d_\mathbb{Y}(y,y')^{q})^{1/q},
\]
which are symmetric in the sense that 
\[
d_{\mathbb{X} \times \mathbb{Y}}((x,y), (x',y')) = d_{\mathbb{Y} \times \mathbb{X}}((y,x), (y',x')).
\]
This property ensures the symmetry of $D_{\otimes}$ even if $\mathbb{X} \neq \mathbb{Y}$. However, one could also use asymmetric distances, as in \cite{nies2025transport}, where for $\alpha \in (0,+\infty)$,
\[
d_{\alpha}((x,y), (x',y')) = \alpha d_\mathbb{X}(x,x') +  d_\mathbb{Y}(y,y').
\]
This can be useful to ensure some specific properties, as we discuss in the next sections. As noted in \citet[Theorem 4.13]{nies2025transport}, $D_{|\bullet}$ can be seen as a limiting case of $D_{\otimes}$ with distance $d_\alpha$ as $\alpha \to + \infty$, which could be potentially useful in some applied settings.
\end{remark}

\begin{remark}
One only needs $d$ to be non-degenerate to ensure that $D_{ \otimes}(X,Y)$ and $D_{|\bullet}(X,Y)$ characterize independence. Thus one could also use a transport cost $d$ that is not a distance, specifically if it is not symmetric or the triangle inequality does not hold. Similarly, one can also consider $\W_{d,p}^p$ instead of $\W_{d,p}$, as in \cite{mordant2022measuring} with $p=2$.
Interestingly, the same work shows that for random variables in $\mathbb{R}^m$, with $d$ the Euclidean distance and $\W = \W_{d,2}$, 
\begin{equation}
\label{def:reference_measure}
\mathcal{W}(\mathscr{L}(X,Y), P  \otimes Q)^2 
- \mathcal{W}(\mathscr{L}(X)  \otimes \mathscr{L}(Y), P \otimes Q)^2
\end{equation}
characterizes independence for any absolutely continuous reference measures $P, Q \in \mathcal{P}(\mathbb{R}^m)$.
The authors suggest to use standard Gaussians, with the potential advantage of reducing to an optimal transport problem where at least one of the two measures is smooth and nicely behaved. This opens interesting questions on the impact of the choice of $P,Q$ on the measure of dependence.
\end{remark}

In the next section we will see that 
$D \in \{D_{ \otimes}, D_{|\bullet}\}$ alone is not informative of how much dependence there is between $X$ and $Y$. However, we highlight some situations where it can be useful per se: 
\begin{itemize}
\item To perform comparisons between two different laws for random vectors. Two variables $(X,Y)$ can be defined more dependent than $(X',Y')$ whenever $D(X,Y) \ge D(X',Y')$. This type of comparison naturally arises, for example, in global sensitivity analysis, where one is interested in ranking input variables according to their influence on the output.
\item To perform independence testing, the null hypothesis being $(X,Y)$ independent, that is,
$D(X,Y) = 0$; see \cite{Ozair2019, warren2021wasserstein, Liu2022}.
\item When estimating a joint law of random variables via loss minimization, as in representation learning, $D$ can be added as a penalization term that forces the variables to be (close to) independent; see \cite{XiaoWang2019}. 
\end{itemize}

\section{Building an index of dependence}
\label{sec:index}

The quantities $D(X,Y)$, with 
$D \in \{D_{ \otimes}, D_{|\bullet}\}$  
defined in the previous section, 
are good tools to measure independence, but by themselves they are not informative about \emph{how much} dependence there is between $X$ and $Y$. To answer this question, we should also understand what is the ``most dependent'' relationship between $X$ and $Y$. As we shall see, there is no definitive and universal answer to this question. 

The road map to build an index of dependence from a measure $D(X,Y)$ of independence is as follows. 
\begin{enumerate}
\item Find $U(X,Y)$ an upper bound of $D(X,Y)$ 
that depends only on the marginal laws of $X$ and $Y$, such that for all joint laws $(X,Y)$, 
\begin{equation}
\label{eq:upper_bound}
D(X,Y) \leq U(X,Y).
\end{equation}
\item Identify the equality cases in~\eqref{eq:upper_bound}, which determine a notion of maximal dependence.
\item Define the index of dependence as 
\begin{equation*}
I(X,Y) = \frac{D(X,Y)}{U(X,Y)}.
\end{equation*}
\end{enumerate}



In such a construction, since $D(X,Y) \geq 0$ and thanks to Lemma~\ref{th:independence}, $I(X,Y) \geq 0$ with equality if and only if $(X,Y)$ are independent. Moreover, $I(X,Y) \leq 1$, with equality if and only if there is equality in~\eqref{eq:upper_bound}. To underline the dependence on $d,p$ we will sometimes write $I(X,Y; d,p)$.

Ideally, the upper bound $U(X,Y)$ should be interpretable and determine equality cases in~\eqref{eq:upper_bound} that reflect our intuition about maximal dependence.
Note that for the same $D(X,Y)$, we may find different upper bounds $U(X,Y)$ which lead to different indices and potentially different notions of maximal dependence.

\subsection{The joint index of dependence}

Let us illustrate an example of this technique with the index $D_\otimes(X,Y)$, defined in~\eqref{defs}, for which finding meaningful upper bounds is a difficult task in general.


\begin{proposition}[Proposition 4.12 and Theorem 5.2 in \cite{nies2025transport}]
\label{prop:nies_extremal}
Let $(\mathbb{X},d_\mathbb{X})$ and $(\mathbb{Y},d_\mathbb{Y})$ be metric spaces and consider 
\[
d((x,y),(x',y')) = d_\mathbb{X}(x,x') + d_\mathbb{Y}(y,y').
\]
Then with $Y'$ an independent copy of $Y$, 
\begin{equation*}
D_{\otimes}(X,Y;d;p)^p \leq \E \left( d_\mathbb{Y}(Y,Y')^p \right),
\end{equation*}
with equality if and only if $Y = f(X)$ a.s. for $f : \mathbb{X} \to \mathbb{Y}$ a $1$-Lipschitz function on the support of $\mathscr{L}(X)$. 
\end{proposition}

The inequality in Proposition~\ref{prop:nies_extremal} was already proved in Theorem 3.1 of \citet{mori2020earth}, but without a proof for the equality case.
The upper bound $\E \left( d_\mathbb{Y}(Y,Y')^p \right)$, when $p=1$ and $\Y = \R$, is \emph{Gini's mean difference} introduced by Gini in 1912; see \citet{yitzhaki2003gini} for a discussion and a historical account. It measures the spread of the distribution of $Y$, and in this sense it can be considered as an alternative to the variance, which corresponds, up to a factor $1/2$, to the case $p=2$. 
When $\Y= \R$, Gini and Ramasubban in the 50s and 60s considered the case of an arbitrary $p \geq 1$, calling it generalized mean difference \citep{yitzhaki2003gini}. As we work here in the context of metric spaces, we call $\E \left( d_\mathbb{Y}(Y,Y')^p \right)$ a generalized mean discrepancy.


In this result, it is crucial that $d$ is the $\ell_1$ norm on the product space, 
and in this case no additional assumptions are needed: $(\mathbb{X},d_\mathbb{X})$ and $(\mathbb{Y},d_\mathbb{Y})$ can be generic metric spaces. 
When $d$ is the $\ell_2$ norm, much less is known about a meaningful upper bound, see Remark~\ref{rmk:upper_bound_W2} below.

Consequently, with $\tilde{U}_{\otimes}(X,Y) = ( \E ( d_\mathbb{Y}(Y,Y')^p ))^{1/p}$ which depends only on $\mathscr{L}(Y)$, we define the index
\begin{equation*}
\tilde{I}_{\otimes}(X,Y) = \frac{D_{\otimes}(X,Y)}{\tilde{U}_{\otimes}(X,Y)}.
\end{equation*}
It satisfies $\tilde{I}_{\otimes}(X,Y) \in [0,1]$, with $\tilde{I}_{\otimes}(X,Y) = 0$ if and only $X$ and $Y$ are independent, and $\tilde{I}_{\otimes}(X,Y) = 1$ if and only if $Y = f(X)$ a.s. for some $1$-Lipschitz function $f$. As detailed in \cite{nies2025transport}, by choosing $d = d_\alpha((x,y),(x',y')) = \alpha d_\mathbb{X}(x,x') + d_\mathbb{Y}(y,y')$ for $\alpha > 0$, the index $\tilde{I}_{\otimes}(X,Y;d_\alpha)$ becomes equal to $1$ if and only if  $Y = f(X)$ for some $\alpha$-Lipschitz function $f$. 

Note that we have broken the symmetry between $X$ and $Y$. If we choose, with the same assumptions, $\tilde{U}_{\otimes}(X,Y) = ( \E ( d_\mathbb{X}(X,X')^p ))^{1/p}$, then the index is equal to $1$ if and only $X = g(Y)$ a.s., for $g : \mathbb{Y} \to \mathbb{X}$ a $1$-Lipschitz function on the support of $\mathscr{L}(Y)$. To restore symmetry, we define
$U_\otimes(X,Y) = \min\{ \E \left( d_\mathbb{X}(X,X')^p \right), \E \left( d_\mathbb{Y}(Y,Y')^p \right)  \}^{1/p}$, 
which yields
\begin{equation*}
I_{\otimes}(X,Y) = \frac{D_{\otimes}(X,Y)}{\min\{ \E \left( d_\mathbb{X}(X,X')^p \right), \E \left( d_\mathbb{Y}(Y,Y')^p \right)\}^{1/p}}.
\end{equation*}
This corresponds to the proposal in \cite{mori2020earth}. Clearly this index is still in $[0,1]$, and is equal to $1$ if and only if $Y = f(X)$ a.s. \emph{or} $X = g(Y)$ a.s. 
for functions $f,g$ which are $1$-Lipschitz on the support of $\mathscr{L}(X)$ and $\mathscr{L}(Y)$ respectively.
This is the definition of $I_{\otimes}$ we will keep in this survey.

Note that if we substitute the $\min$ by a $\max$ in the denominator, we rather obtain an index which is equal to $1$ if and only if $Y = f(X)$ a.s. \emph{and} $X = g(Y)$ a.s. for functions $f,g$ which are $1$-Lipschitz on the supports of $\mathscr{L}(X)$ and $\mathscr{L}(Y)$, meaning that, on the support of $\mathscr{L}(X)$, $f$ preserves distances and $f=g^{-1}$. 
Up to our knowledge, this normalization has not been considered in the literature.

An alternative proposal in \cite{nies2025transport} is to 
scale the distance function on the product space with the generalized mean discrepancies, that is, 
\begin{equation*}
d_{\mathscr{L}(X),\mathscr{L}(Y)}((x,y),(x',y')) 
= \frac{d_\mathbb{X}(x,x')}{\E \left( d_\mathbb{X}(X,X')^p \right)^{1/p}} + \frac{d_\mathbb{Y}(y,y')}{\E \left( d_\mathbb{Y}(Y,Y')^p \right)^{1/p}},
\end{equation*}
and define the index as $D_{\otimes}(X,Y;d_{\mathscr{L}(X),\mathscr{L}(Y)};p)$, which 
can be proved to be in $[0,1]$ with 
well-characterized extremes
\cite[Proposition 6.6]{nies2025transport}.

\subsection{The conditional index of dependence}

Next we discuss the same technique with the index $D_{|\bullet}$, defined in~\eqref{defs}. Interestingly, for random variables in $\R$ endowed with the Euclidean distance and $p=1$, 
this proposal can be traced back to works of Corrado Gini between 1911 and 1914,
as described in \cite{CifarelliRegazzini2017}.

\begin{proposition}[Section 9 in \cite{CifarelliRegazzini2017}, Theorem 2.2 in \cite{wiesel2022measuring}, Theorem 1 in \cite{borgonovo2024global}]
With $Y'$ an independent copy of $Y$, 
\begin{equation*}
D_{|\bullet}(X,Y;d_\Y,p)^p \leq \E  \left(d_\mathbb{Y}(Y,Y')^p \right) ,
\end{equation*}
with equality if and only if $Y = f(X)$ a.s. for a measurable function $f : \mathbb{X} \to \mathbb{Y}$. 
\end{proposition}

The upper bound is similar to the one for $D_\otimes$: it is again the the generalized mean discrepancy of the variable $Y$.
As emphasized in \cite{borgonovo2024global}, the key mathematical property to prove this proposition is the strict convexity of  
$t \mapsto \W_{d_\Y,p}^p(P,(1-t) \delta_{y_1} + t \delta_{y_2})$ when $y_1 \neq y_2$; $d_{\mathbb{Y}}$ does not need to be a distance for this to hold.
As a consequence, with $U_{|\bullet}(X,Y) = ( \E ( d_\mathbb{Y}(Y,Y')^p ))^{1/p}$, we define the index
\begin{equation*}
I_{|\bullet}(X,Y) = \frac{D_{|\bullet}(X,Y)}{U_{|\bullet}(X,Y)}.
\end{equation*}
It satisfies $I_{|\bullet}(X,Y) \in [0,1]$, with $I_{|\bullet}(X,Y) = 0$ if and only $X$ and $Y$ are independent, and $I_{|\bullet}(X,Y) = 1$ if and only if $Y = f(X)$ a.s. for some measurable function $f$. 


\subsection{A more principled way}
\label{sec:principled}

%
%
The upper bounds $U(X,Y)$ described above are \emph{ad hoc} for the type of distances involved. A more canonical way would be to define $U(X,Y)$ as the supremum over all possible couplings
\begin{equation}
\label{eq:definition_U_sup}
U(X,Y) = \sup_{\mathscr{L}(X',Y') \in \Gamma(X,Y)} D(X',Y'),
\end{equation}
for $D \in \{D_{ \otimes}, D_{|\bullet}\}$
which is for instance the proposal of \cite{mordant2022measuring}. In the case $D = D_\otimes$, the supremum is attained as $\Gamma(X,Y)$ is compact for the same topology making $D_\otimes$ continuous \citep[Section 2]{mordant2022measuring}. In the case $D = D_{|\bullet}$, the situation is more delicate as $\Gamma(X,Y)$ may not be compact for a topology making $D_{|\bullet}$ lower semi-continuous. 
By construction $D \leq U$, and, provided the supremum is attained, for any $(X,Y)$, there exists $(X',Y')$ such that $X' \eqd X$, $Y' \eqd Y$ and $I(X,Y) = 1$. However, the main difficulty in this case is to characterize the couplings yielding maximal dependence, that is, realizing the supremum.  


Note that both $U_{|\bullet}(X,Y)$ and $U_\otimes(X,Y)$ as defined above do not fit the definition in \eqref{eq:definition_U_sup}.
Indeed, take $X,Y$ Bernoulli random variables with parameters $p_X,p_Y \in (0,1)$ and $p_X \neq p_Y$. Then for every measurable function $f$, $f(X)$ has a different distribution than $Y$. Thus for every coupling $(X',Y')$, $I_{|\bullet}(X',Y') < 1$ and $I_{\otimes}(X',Y') < 1$. In contrast, if $I$ is built by normalizing with $U$ as in~\eqref{eq:definition_U_sup} by construction there exists $(X',Y')$ such that $I(X',Y') = 1$. 


\begin{remark}
\label{rmk:upper_bound_W2}
To illustrate the difficulty in computing the supremum in~\eqref{eq:definition_U_sup}, consider the case $\mathbb{X} = \mathbb{Y} = [0,1]$, and $d((x,y), (x',y')) = \sqrt{|x-x'|^2 + |y-y'|^2}$ is the Euclidean distance, that is, we take the norm $\ell_2$ in the product space. Let us take $X,Y$ uniform random variables over $[0,1]$. 
We denote by $\lambda$, $\lambda^2$ respectively the uniform measure on $[0,1]$ and $[0,1]^2$, so that $X \sim \lambda$, $Y \sim \lambda$ and $(X,Y) \sim \lambda^2$ if $X,Y$ are independent. To our knowledge, the computation of the supremum in~\eqref{eq:definition_U_sup} is still an open question when $D = D_\otimes(\cdot;d,2)$. In other terms, with $\gamma = \mathscr{L}(X,Y)$, the computation of 
\begin{equation*}
\sup_\gamma \left\{ \W(\gamma,\lambda^2) \ : \ \gamma \in \Gamma(\lambda,\lambda) \right\},
\end{equation*}
where $\W$ is the classic quadratic Wasserstein distance over $[0,1]^2$, has not been solved. 
A natural conjecture is that the supremum is reached when $X = Y$, or $X = 1-Y$ (that is, $\gamma$ concentrated on one of the two diagonals of $[0,1]^2$), but a proof is lacking. The difficulty in solving such a seemingly simple problem has been noted in \cite{mordant2022measuring,DeKeyserGijbels2025}.
Note that Proposition~\ref{prop:nies_extremal} does not solve the problem, as the cost in the latter is the Manhattan distance $d((x,y), (x',y')) = |x-x'| + |y-y'|$ instead of the Euclidean distance.
It illustrates how a small change in the cost function greatly impacts the search for upper bounds. 
\end{remark}

\section{Analytic expressions}
\label{sec:analytic}

An issue with optimal transport is that $\W$ usually lacks an analytic expression. This is because, although an \emph{optimal coupling} that realizes the infimum in \eqref{def:wass} is known to exist, its expression is not known in general. We comment on a few cases where such an expression is available. 

\subsection{One dimensional space}

In one dimension $\W_p(\mathscr{L}(X),\mathscr{L}(Y))$ coincides with the $L^p$ distance between the inverse cumulative distribution functions, a.k.a. the quantile functions. In particular, if $\mathbb{Y} = \mathbb{R}$, with $F_Y(t) = \mathbb{P}(Y \leq t)$ and $F_{Y|X}(t) = \mathbb{P}(Y \leq t | X)$ the cumulative distribution functions of $Y$ and $Y|X$, 
\begin{equation}
\label{eq:I_cond_cdf}
D_{|\bullet}(X,Y) = \left( \E_{X} \left( \int_{0}^{1} \left|F_{Y|X}^{-1}(t) - F_{Y}^{-1}(t) \right|^p dt \right) \right)^{1/p}.
\end{equation}
This formula is useful to the extent that the quantile functions and integrals are tractable.

On the other hand, even if $X$ and $Y$ are real-valued random variables, computing $D_\otimes(X,Y)$ implies solving an optimal transport problem between two-dimensional probabilities $\mathscr{L}(X,Y)$ and $\mathscr{L}(X) \otimes \mathscr{L}(Y)$, 
for which we do not have explicit formulas involving the quantile functions. One typically needs to estimate $D_\otimes$ with samples and solve the optimal transport problem numerically as discussed below. 

\subsection{The Gaussian case}

%
If $X,Y$ are random Gaussian vectors in $\R^m$ of means $a_X, a_Y$ and covariances $\Sigma_X, \Sigma_Y$, and $d$ is the Euclidean distance, then $\W_{d,2}^2(\mathscr{L}(X),\mathscr{L}(Y))$ equals
\begin{equation*}
\| a_X - a_Y \|^2 + \mathrm{tr} \left( \Sigma_X + \Sigma_Y - 2 (\Sigma_X^{1/2} \Sigma_Y \Sigma_X^{1/2})^{1/2}  \right),
\end{equation*}
where $\mathrm{tr}$ denotes the trace of a matrix and $A^{1/2}$ the positive semi-definite symmetric square root of a symmetric matrix $A$ \citep{dowson1982frechet,gelbrich1990formula}.

When $d$ is Euclidean and $p=2$, this result can be used for an analytical expression of $I_{|\bullet}(\cdot;d,p)$ if $\mathscr{L}(Y|X)$ is Gaussian, see e.g. \citet[Proposition 3]{borgonovo2024global}.

In principle, it can also be used for an analytical expression of $D_\otimes$, however one needs $d$ to be the Euclidean distance, which is \emph{not} the assumption of Proposition~\ref{prop:nies_extremal}. Thus, even if $\mathbb{X} = \mathbb{Y} = \mathbb{R}$ and $(X,Y)$ is a Gaussian vector, there is no known analytical expression for the index $I_\otimes(X,Y)$ proposed in \cite{mori2020earth,nies2025transport}.



On the other hand, the proposal of \cite{mordant2022measuring} when $(X,Y)$ is a Gaussian vector, is to use $D_\otimes(X,Y;d,p=2)^2$ with $d$ the Euclidean distance.
In this case they are able to solve Problem~\eqref{eq:definition_U_sup} 
when restricting the supremum to 
Gaussian couplings in $\Gamma(X,Y)$. 
With this computation they can normalize $D_\otimes^2$ and obtain an index of dependence $I_G$ that can be computed explicitly and characterizes independence for Gaussian vectors.
Moreover, it is possible to identify the cases where it is equal to $1$, which determines a notion of maximal dependency between Gaussian vectors.

If $(X,Y)$ is not a Gaussian vector, \cite{mordant2022measuring} propose to compute the index $I_\text{G}(\tilde{X},\tilde{Y})$ with $(\tilde{X}, \tilde{Y})$ the Gaussian vector having the same covariance matrix as $(X,Y)$. However, if $I_\text{G}(\tilde{X},\tilde{Y}) = 0$, it means that $\tilde{X}$ and $\tilde{Y}$ are independent, but not necessarily that $X$ and $Y$ are independent. This is not a surprise: by using only moments of order two of $(X,Y)$, it is not possible to characterize independence.

\subsection{The simplest bivariate random vector}
\label{sec:bivariate}


To illustrate some 
topics of this section,
take $(X,Y)$ Gaussian random vector in $\R^2$ with covariance matrix
\renewcommand*{\arraystretch}{1.2}
\[
\Sigma_{(X,Y)} = 
\begin{pmatrix}
1  & \rho \\
\rho  & 1
\end{pmatrix}.
\]
In this case $\rho \in [-1,1]$ coincides with the linear correlation of $X$ and $Y$. Using e.g.~\eqref{eq:I_cond_cdf} we have, for $p=2$,  
\begin{equation*}
I_{|\bullet}(X,Y;p=2) = 1 - \sqrt{1-\rho^2},
\end{equation*}
\citet[Lemma 3.7]{wiesel2022measuring}. 
The index $I_\text{G}(X,Y)$ has been computed in \citet[Example 3]{mordant2022measuring}:
\begin{equation*}
I_\text{G}(X,Y) =\frac{2 - \sqrt{1+\rho} - \sqrt{1-\rho}}{2- \sqrt{2}}.
\end{equation*}
Regarding $D_\otimes(X,Y)$, when the ground distance $d$ is not Euclidean, it is difficult to compute it explicitly. For $d((x,y), (x',y')) = |x-x'|+|y-y'|$, \citet[Theorem 3.5]{mori2020earth} obtain upper and lower bounds 
\begin{equation*}
|1- \sqrt{1-\rho}| \leq I_{\otimes}(X,Y;p=1) \leq \sqrt{ 1 - \sqrt{1-\rho^2}}.
\end{equation*}

The graphs of these expressions are reported in Figure~\ref{fig:bivariate}. Note that all the indices detect an increasing correlation with $|\rho|$, which is equal to $1$ if and only if $\rho = \pm 1$. The behaviour of $I_{|\bullet}$ and $I_G$ is very similar and smooth around the origin, whereas the upper and lower bounds of $I_\otimes$ show a singularity that makes it more sensitive to detecting departures from independence.

\begin{figure}
\centering
\begin{tikzpicture}
\begin{axis}[
xmin = -1, xmax = 1.,
ymin = 0., ymax = 1.,
xlabel={$\rho$},
ylabel={$I$},
legend style={at={(0.1,0.99)},anchor=north west}
]
\addplot[line width = 1pt,color=red,samples=100,domain=-1:1]{1-sqrt(1-x^2)};
\addlegendentry{{\small $I_{|\bullet}$},  p=2};
\addplot[line width = 1pt,color=blue,samples=100,domain=-1:1]{( 2 - sqrt(1+x) - sqrt(1-x) )/(2 - sqrt(2)) };
\addlegendentry{{\small $I_G$},  p=2};
\addplot[line width = 1pt,color=black!60!green,samples=100,style = dotted,domain=-1:1]{abs(1 - sqrt(1-x))};
\addlegendentry{{\small l.b. $I_{\otimes}$, p=1}};
\addplot[line width = 1pt,color=black!60!green,samples=100,style = dashed,domain=-1:1]{sqrt(1-sqrt(1-x^2))};
\addlegendentry{{\small u.b. $I_{\otimes}$, p=1}};
\end{axis}
\end{tikzpicture}
\caption{Values of the different indices (lower and upper bound in the case of $I_\otimes$) for a bivariate Gaussian with correlation coefficient $\rho$, see Section~\ref{sec:bivariate}.}
\label{fig:bivariate}
\end{figure}
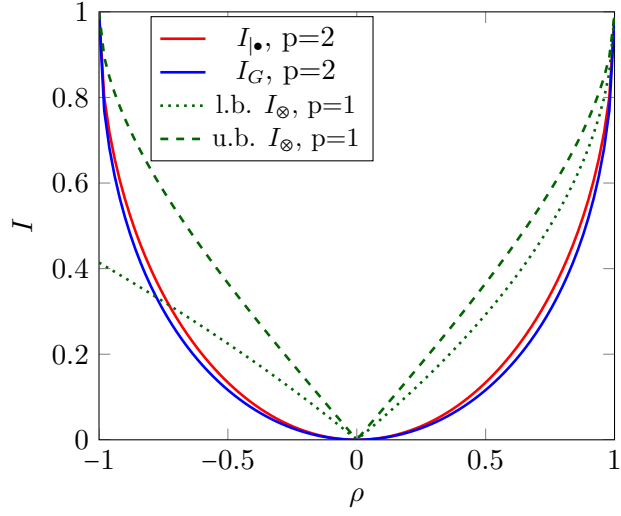


\section{Invariances}
\label{sec:invariances}

From the way they are built, these indices enjoy certain invariances, in the sense that $I(X,Y)$ may coincide with $I(X',Y')$ when $X' = f(X)$ and $Y'= g(Y)$ for functions $f,g,$ belonging to certain classes. The most common classes are measurable bijections, isometries, similarities, and increasing transformations, where we recall that a function $f : \mathbb{X} \to \mathbb{X}$ is
\begin{itemize}
\item a measurable bijection if $f$ is bijective and 
both $f$ and its inverse $f^{-1}$ are measurable.
\item an isometry if $f$ is bijective and $d_\mathbb{X}(f(x),f(x')) = d_\mathbb{X}(x,x')$ for all $x,x'$; 
\item a similarity transformation if, for some $\lambda > 0$, $f$ is bijective and $d_\mathbb{X}(f(x),f(x')) = \lambda d_\mathbb{X}(x,x')$ for all $x,x'$, equivalently if $f$ is an isometry from $(\mathbb{X},\lambda d_\mathbb{X})$ to $(\mathbb{X},d_\mathbb{X})$; in this case we call $\lambda$ the scaling factor;
\item a bijective increasing transformation 
if $\X$ is an interval of $\R$, and $f$ is bijective and increasing. 
\end{itemize}
If $\mathbb{X} = \mathbb{Y} = \R^m$ endowed with the Euclidean distance, then $f$ is an isometry if and only if $f(x) = Ax + b$ for an orthogonal matrix $A$ and a vector $b \in \R^m$, and $f$ is a similarity if and only if $f(x) = \lambda Ax + b$, again with $A$ orthogonal matrix and $b \in \R^m$, being $\lambda$ the scaling factor.

\begin{definition}
Let $\mathcal{F}$ be a subset of pairs $(f,g)$ of measurable bijections of $\mathbb{X}$ and $\mathbb{Y}$. We say that $I$ is $\mathcal{F}$-invariant if, for all $(f,g) \in \mathcal{F}$, and all pairs $(X,Y)$ of random variables, 
\begin{equation*}
I(f(X),g(Y)) = I(X,Y).
\end{equation*}
\end{definition}

\noindent Note that for simplicity we restrict to invertible transformations, and from the definition we can see that the invariance $I(f(X),g(Y)) = I(X,Y)$ extends to the subgroup of pairs of measurable bijections generated by $\mathcal{F}$. In particular, without loss of generality, we can assume $(f^{-1},g^{-1}) \in \mathcal{F}$ if $(f,g) \in \mathcal{F}$. 



The group of isometries is important as it leaves the Wasserstein distance invariant: if $X,X'$ are random variables valued in $\mathbb{X}$, then $\W_{d_\X,p}(\mathscr{L}(f(X)),\mathscr{L}(f(X'))) = \W_{d_\X,p}(\mathscr{L}(X), \mathscr{L}(X'))$ for any isometry $f$ of $(\mathbb{X},d_\mathbb{X})$. Armed with this observation, it is easy to prove the following. Variants of the  
result below has already been proved: see Theorem 3.8 in \cite{mori2020earth}, Proposition 6.2 in 
\cite{nies2025transport}, Proposition 1 in \cite{mordant2022measuring} for the joint index; and Lemma 3.2 in \cite{wiesel2022measuring}, Theorem 2 in \cite{borgonovo2024global} for the conditional index.

\begin{proposition}[Invariances]~~
\begin{enumerate}
\item If $f$ and $g$ are similarities of $(\mathbb{X},d_\mathbb{X})$ and $(\mathbb{Y},d_\mathbb{Y})$ respectively, with the same scaling factor, and if $\X \times \Y$ is equipped with a $\ell_q$ norm as in Remark~\ref{rem:symmetry},
\begin{equation*}
I_\otimes(f(X),g(Y)) = I_\otimes(X,Y).
\end{equation*}
\item If $f$ is a measurable bijection of $\mathbb{X}$ and $g$ is a similarity of $(\mathbb{Y},d_\mathbb{Y})$
\begin{equation*}
I_{|\bullet}(f(X),g(Y)) = I_{|\bullet}(X,Y).
\end{equation*}
\end{enumerate}
\end{proposition}

In other terms, the group of invariance $\mathcal{F}$ for the joint index contains all pairs of similarities with common scaling factors,
while the one for the conditional index contains all pairs of measurable bijections and similarities. 
In particular, the conditional index is asymmetric also in the type of invariance it induces. The class of invariance is related to the class of maximal dependence $\mathcal{I}_1$.


%
\begin{proposition}
\label{prop:invariance_optimal}
Let $\mathcal{I}_1$ be the set of $h : \mathbb{X} \to \mathbb{Y}$ such that $I(X,h(X)) = 1$ for any random variable $X$, and assume $I$ is $\mathcal{F}$-invariant. Then $g \circ h \circ f^{-1} \in \mathcal{I}_1$ for every $h \in \mathcal{I}_1$ and $(f,g) \in \mathcal{F}$.
\end{proposition}

\begin{proof}
Assume $Y = g(h(f^{-1}(X)))$ a.s.. Then with $X' = f^{-1}(X)$ and $Y' = g^{-1}(Y)$ we have $Y' = h(X')$ a.s.. Thus $1=I(X',Y') = I(f(X'),g(Y')) = I(X,Y)$. 
\end{proof}


\begin{remark}
As discussed in \citet[Remark 2]{mori2019four}, ``More invariance is not necessarily better'', since Proposition~\ref{prop:invariance_optimal} shows that it makes $\mathcal{I}_1$ larger, with potentially too many couplings labelled as maximally dependent. Moreover, invariance also has important consequences on the continuity of the index, as it can prevent the index to be continuous with respect to the convergence in distribution of $(X,Y)$; see \citet[Theorem 1]{mori2019four} and the discussion in Section~\ref{sec:continuity_conditional} below.
\end{remark}

\begin{remark}
Assuming that $\mathbb{X} = \mathbb{Y}$ and $I(X,X) = 1$ for all random variables (a quite mild assumption), we see that $g \circ f^{-1}$ belongs to the set $\mathcal{I}_1$ of maximal dependent functions
for all $(f,g) \in \mathcal{F}$. Moreover, if $\mathbb{X}$ is a linear space and $\mathcal{F}$ contains at least all pairs of isometries, then $I(X,-X) = 1$. It illustrates the difficulty to build a signed index of dependence if we have such class of invariance,
since $(X,-X)$ is arguably a natural notion of negative dependence.
\end{remark}

In the case $\X = \Y = \R$, it is also natural to require $I$ to be  
invariant by increasing transformations, that is, $I(f(X),g(Y)) = I(X,Y)$ for every $f,g$ bijective increasing transformations. 
This can be achieved for absolutely continuous random variables by defining the index at the level of the copulas: for an index $I$
we define 
\[ \bar{I}(X,Y) = I(F_X(X),F_Y(Y)), \]
where $F_X$, $F_Y$ are the cumulative distribution functions of $X,Y$ respectively. If $X$ and $Y$ are absolutely continuous, marginally $F_X(X)$ and $F_Y(Y)$ are 
uniformly distributed on $[0,1]$, and the joint distribution of $(F_X(X),F_Y(Y))$ is called a copula. By construction $\bar{I}(X,Y) = \bar{I}(f(X),g(Y))$ for every $f,g$ bijective increasing transformations, 
as $(f(X),g(Y))$ and $(X,Y)$ share the same copula. 
For example, \cite{Marti2017, BeneventoDurante2024} and \cite{DeKeyserGijbels2025} define Wasserstein-based indices of dependence directly at the level of the copulas. The latter also studies copulas 
of more than two random variables, as discussed in Section~\ref{sec:more_than_2} . 

In the case $\X = \Y = \R^m$, one possible generalization of copulas relies on optimal transport and 
showcases yet another use of optimal transport in the context of measuring dependence. Taking $P$ the uniform distribution over 
the unit cube $[0,1]^m$ of $\R^m$, we denote by $T_X$, $T_Y$ the optimal transport maps from respectively $\mathscr{L}(X)$ and $\mathscr{L}(Y)$ to $P$. These transport maps play the role of ``multivariate ranks'', analogue to quantile functions \citep{chernozhukov2017monge,ghosal2022multivariate}. Indeed, marginally $T_X(X)$ and $T_Y(Y)$ are 
uniformly distributed on $[0,1]^m$, and if $m=1$, we recover the 
copula formulation as $F_X = T_X$ and $F_Y = T_Y$. For an index $I$, we can define 
\[ \bar{I}(X,Y) = I(T_X(X),T_Y(Y)), \]
which has been the proposal of \citet{deb2024distribution}, being $I$ an kernel-based index proposed by the same authors \citep{deb2020measuring}. In this multi-dimensional setting, often the interest is not to guarantee some invariances of $\bar{I}$, but rather to ensure that 
when estimating $\bar{I}$ via a plug-in empirical estimator, 
the distribution of the estimator does not depend on $\mathscr{L}(X)$ and $\mathscr{L}(Y)$ if $(X,Y)$ are independent. 



\section{Robustness}
\label{sec:robustness}
By robustness, we mean understanding how the index changes as we change its inputs, that is, the random variables $(X,Y)$. Note that invariances, investigated in the previous sections, can already be seen as a form of robustness, as they guarantee that the index does not change if some transformation is applied to $(X,Y)$. We emphasize that both indices depend only on the law of $(X,Y)$ as an input. 
Thus in this section we will sometimes write $I(P)$ for $P = \mathscr{L}(X,Y)$ instead of $I(X,Y)$ for clarity.

\subsection{Continuity of the joint index of dependence in weak topology}

Recall that a sequence $(X_n,Y_n)_{n \geq 1}$ converges in distribution to $(X,Y)$ if $\mathscr{L}(X_n,Y_n)$ converges weakly to $\mathscr{L}(X,Y)$ as $n \to + \infty$, or equivalently if $\E(f(X_n,Y_n))$ converges to $\E(f(X,Y))$ for every continuous bounded function $f : \mathbb{X} \times \mathbb{Y} \to \mathbb{R}$. If this holds and if $\mathbb{X}$, $\mathbb{Y}$ are Euclidean spaces, we say 
there is convergence of $p$-moments if $\E(\| X_n \|^p+ \| Y_n \|^p)$ converges to $\E(\| X \|^p+\| Y \|^p)$.
These notions can be easily generalized to arbitrary metric spaces  $\mathbb{X}$, $\mathbb{Y}$ by requiring $\E(d_\mathbb{X}(X_n,x_0)^p+d_\mathbb{Y}(Y_n,y_0)^p)$ to converge to $\E(d_\mathbb{X}(X,x_0)^p+d_\mathbb{Y}(Y,y_0)^p)$, where $x_0,y_0$ are given points in $\mathbb{X}$, $\mathbb{Y}$. It is not difficult to show that the convergence in distribution together with convergence of $p$-moments does not depend on the reference points $x_0$, $y_0$, and together they play an important role for the Wasserstein distance, since it can be shown to metrize this joint notion of convergence. We refer e.g. to \citet[Chapter 6]{villani2009optimal} for a thorough discussion on this topic. 

In this section we explain why
convergence in the Wasserstein distance implies convergence of the joint index of dependence. The upper bounds $U_{\otimes}(X,Y)$ and $U_{|\bullet}(X,Y)$ depend on $\E(d_\mathbb{X}(X,X')^p)$ and $\E(d_\mathbb{Y}(Y,Y')^p)$ the generalized mean discrepancies. Adapting ideas of the proof of \citet[Theorem 4.1]{wiesel2022measuring}, they are easily shown to be 
Lipschitz, and a fortiori continuous, with respect to $\W$.

\begin{lemma}
\label{lm:continuity_GMD}
If $X_1, X_2$ are random variables valued in 
a metric space $(\X, d_{\mathbb{X}})$, and $X_1', X_2'$ are independent copies of $X_1, X_2$ respectively,
\begin{equation*}
\left| \left(\E(d_\mathbb{X}(X_1,X_1')^p) \right)^{1/p} - \left(\E(d_\mathbb{X}(X_2,X_2')^p) \right)^{1/p} \right| 
\leq 2 \, \W_{d_\mathbb{X},p}(\mathscr{L}(X_1),\mathscr{L}(X_2)).
\end{equation*}
\end{lemma}

\begin{proof}
Without loss of generality we can assume $(X_1,X_2)$ with law given by 
an optimal transport coupling between $\mathscr{L}(X_1)$ and $\mathscr{L}(X_2)$, and $(X_1',X_2')$ to be an independent copy of $(X_1,X_2)$. 
Denote by $d_{L^p(\mathbb{X})}$ the $L^p$ distance on $\X$-valued random variables. The triangle inequality yields
\begin{equation*}
\left| d_{L^p(\mathbb{X})}(X_1,X_1') - d_{L^p(\mathbb{X})}(X_2,X_2') \right| 
 \leq d_{L^p(\mathbb{X})}(X_1,X_2) + d_{L^p(\mathbb{X})}(X'_1,X'_2),
\end{equation*}
which is the claim as $d_{L^p(\mathbb{X})}(X_1,X_2) = d_{L^p(\mathbb{X})}(X'_1,X'_2) = \W_{d,p}(\mathscr{L}(X_1),\mathscr{L}(X_2))$. 
\end{proof}



This yields the following result.

\begin{proposition}
\label{prop:continuity_d_joint}
Assume $(X_n,Y_n)_{n \geq 1}$ converges to $(X,Y)$ in distribution together with convergence of the $p$-moments,
i.e. $\mathcal{W}_p(\mathscr{L}(X_n,Y_n),\mathscr{L}(X,Y))$ converges to $0$. Then, 
\begin{equation*}
\lim_{n \to + \infty} D_\otimes(X_n,Y_n) = D_\otimes(X,Y). 
\end{equation*}
If in addition $U_\otimes(X,Y) > 0$,
\begin{equation*}
\lim_{n \to + \infty} I_\otimes(X_n,Y_n) = I_\otimes(X,Y).
\end{equation*}
\end{proposition}

Note that $U_\otimes(X,Y) > 0$ if both $X$ and $Y$ are not deterministic. 
When $d = d_{\mathbb{X}} + d_{\mathbb{Y}}$, we can even obtain a Lipschitz continuity result of $D_\otimes(X,Y)$ as a function of $\mathscr{L}(X,Y)$ endowed with the Wasserstein distance \citep[Theorem 4.8]{nies2025transport}.

An important consequence, 
which we analyze more in 
detail in the next section, concerns the approximation with i.i.d. samples. Indeed, fix a law $\mathscr{L}(X,Y)$ and consider a sequence $(X_n,Y_n) \simiid \mathscr{L}(X,Y)$. Let us form the empirical distribution 
\begin{equation}
\label{eq:empirical_distribution}
\hat P_{\mathscr{L}(X,Y)}= \frac{1}{n} \sum_{i=1}^n \delta_{(X_i,Y_i)}.
\end{equation}
Then by Proposition~\ref{prop:continuity_d_joint} and the Glivenko-Cantelli theorem, 
$I_\otimes(\hat P_{\mathscr{L}(X,Y)})$ converges a.s. to $I(X,Y)$ as $n \to + \infty$, at least provided the $p$-th moment of $X$ and $Y$ are finite.


\subsection{Contamination with mixture}

A more specific setting to test the robustness of the index is the contamination model. Fix $\mathscr{L}(X,Y)$, $0\le \varepsilon \le 1$, and a contamination distribution $\mathscr{L}(\tilde{X},\tilde{Y}) \in \Gamma(X,Y)$ that has the same marginal distributions as $\mathscr{L}(X,Y)$. 
%
Define $(X_\varepsilon,Y_\varepsilon)$ a pair of random variables with law
\begin{equation*}
(1-\varepsilon) \mathscr{L}(X,Y) + \varepsilon \mathscr{L}(\tilde{X},\tilde{Y}).
\end{equation*}
If both $(X,Y)$ and $(\tilde{X},\tilde{Y})$ have finite $p$-moments then the continuity result of Proposition~\ref{prop:continuity_d_joint} yields that $I_\otimes(X_\varepsilon,Y_\varepsilon)$ converges to $I_\otimes(X,Y)$ as $\varepsilon \to 0$. We can say more: as noted in~\citet[Proposition 4.4]{nies2025transport}, by convexity of the Wasserstein distance, 
\begin{equation*}
D_\otimes(X_\varepsilon,Y_\varepsilon)^p \leq (1-\varepsilon)D_\otimes (X,Y)^p + \varepsilon D_\otimes(\tilde{X},\tilde{Y})^p.
\end{equation*}
In the particular case where $(\tilde{X},\tilde{Y})$ are independent, then $D_\otimes(X_\varepsilon,Y_\varepsilon) \leq (1-\varepsilon)D_\otimes (X,Y)$ and, as $U_\otimes$ is left unchanged, $I_\otimes(X_\varepsilon,Y_\varepsilon) \leq (1-\varepsilon)I_\otimes (X,Y)$: contaminating $(X,Y)$ 
with independent random variables decreases the index of dependence.

\subsection{Continuity of the conditional index in adapted topology}
\label{sec:continuity_conditional}

The conditional index of dependence requires the conditional distribution $\mathscr{L}(Y|X)$, that is, a disintegration of the measure $\mathscr{L}(X,Y)$, and it is much less robust. 

\begin{example}
\label{ex:iidcond}
Assume $(X,Y)$ is a pair of independent random variables with uniform marginal distributions
on $\mathbb{X} = \mathbb{Y} = [0,1]$. Then $I_{|\bullet}(X,Y) =  0$. Consider a sequence $(X_n,Y_n) \simiid \mathscr{L}(X,Y)$ and take $\hat P_{\mathscr{L}(X,Y)}$ the empirical distribution defined in~\eqref{eq:empirical_distribution}. Then
$I_{|\bullet}(\hat P_{\mathscr{L}(X,Y)}) = 1$ a.s. Indeed, 
as $(X_n)_{n \geq 1}$ are all distinct almost surely, we can always find a measurable function $f : [0,1] \to [0,1]$ such that $Y_i = f(X_i)$ for all $i$. Thus for this particular function $f$, we have $\hat Y=f(\hat X)$ a.s. if $(\hat X,\hat Y) \sim \hat P_{\mathscr{L}(X,Y)}$. Thus $I_{|\bullet}(\hat{X},\hat{Y}) = 1$.
\end{example}


The reason for the failure of continuity in the previous example is that $\hat P_{\mathscr{L}(X,Y)}$ approximates $\mathscr{L}(X,Y)$ well in the sense of weak convergence, whereas its conditional distribution is not a good approximation of $\mathscr{L}(Y|X)$ in the same topology. The same argument as in Example~\ref{ex:iidcond} 
shows that an index that is equal to $0$ for the independent case, and to $1$ if $Y$ is a measurable function of $X$, cannot be continuous for the topology of weak convergence.


The good notion of topology to capture this instability is the \emph{adapted topology}, which can be traced back to the works of \cite{aldous} and \cite{hellwig1996sequential}; we refer to \citet{backhoff2020all} and references therein for a thorough presentation and many equivalent definitions.
It is the coarsest topology on the set of laws of random variables which makes the law $(X,\mathscr{L}(Y|X))$ continuous as a function of the law of $(X,Y)$. It is finer than the topology of weak convergence, in the sense that if $\mathscr{L}(X_n,Y_n)$ converges in adapted topology to $\mathscr{L}(X,Y)$, then it also converges weakly, but the contrary is not true. 

This topology, together with the convergence of $p$-moments, is metrized by the so-called \emph{adapted Wasserstein distance} \citep[Theorem 1.3]{
backhoff2020all}. 
If $P, Q$ are two probability distributions on a product space $\mathbb{X} \times \mathbb{Y}$ with $(X,Y) \sim P$ and $(X',Y') \sim Q$, we write $P_1 = \mathscr{L}(X)$ and $Q_1 = \mathscr{L}(X')$ the first marginals of $P,Q$. The adapted Wasserstein distance of order $p$ is 
\begin{equation*}
\mathcal{AW}_{p}(P,Q)^p =  \inf_{\gamma \in \Gamma(P_1,Q_1)} \E_{(X,X') \sim \gamma} \left( d_\mathbb{X}(X,X')^p  \right.
\left.+ \W_{d_{\mathbb{Y}},p} \left( \mathscr{L}(Y|X), \mathscr{L}(Y'|X') \right)^p \right)   .
\end{equation*}
From the definition it follows that this distance requires choosing a coupling $\gamma$ of $(X,X')$, and, given $(X,X') = (x,x')$, computing a coupling of $\mathscr{L}(Y|X=x)$ and $\mathscr{L}(Y'|X=x')$. It takes explicitly in account the distance between conditional distributions.

\begin{proposition}[Theorem 4.1 in \cite{wiesel2022measuring}]
For $p=1$ and $\X = \Y$ any Polish space,
assume that $\mathscr{L}(X_n,Y_n)_{n \geq 1}$ converges to $\mathscr{L}(X,Y)$, in adapted topology together with convergence of the first moment, 
i.e. $\mathcal{AW}_1(\mathscr{L}(X_n,Y_n),\mathscr{L}(X,Y))$ converges to $0$. Then 
\begin{equation*}
\lim_{n \to + \infty} D_{|\bullet}(X_n,Y_n) = D_{|\bullet}(X,Y).  
\end{equation*}
If in addition $U_{|\bullet}(X,Y) > 0$,
\begin{equation*}
\lim_{n \to + \infty} I_{|\bullet}(X_n,Y_n) = I_{|\bullet}(X,Y).
\end{equation*}
\end{proposition}

Actually, \cite{wiesel2022measuring} shows that that $D_{|\bullet}$ is Lipschitz as a function of $\mathscr{L}(X,Y)$
endowed with the adapted Wasserstein distance.

\section{Sample complexity}
\label{sec:sample_comp}

In applied settings, one typically does not have access to $\mathscr{L}(X,Y)$ but rather observes an i.i.d. sample $(X_1,Y_1),\dots, (X_n,Y_n) \simiid \mathscr{L}(X,Y)$. It becomes important to be able to approximate the indices $I$ through samples and for the approximation to converge fast as the sample size $n$ increases. For simplicity in this section we assume that $\mathbb{X} = \mathbb{Y} = [0,1]^m$ endowed with the Euclidean distance.

The denominators $U_{\otimes}$ and $U_{|\bullet}$ in the index depend on the generalized mean discrepancies $\E(d_\X(X,X')^p)$ and $\E(d_\Y(Y,Y')^p)$. They can easily be estimated in an unbiased way with $U$ statistics \citep[Chapter 12]{van2000asymptotic}. For instance the quantity $\E(d_\X(X,X')^p)$ is estimated by $1/(n(n-1)) \sum_{i \neq j} d_\mathbb{X}(X_i,X_j)^p$. It explains why the focus is usually on the estimation of the numerators $D_{\otimes}$ and $D_{|\bullet}$. Nevertheless, a control of the dependence between the estimation of the numerator and the denominator, which cannot always be found in the literature, would be needed to deduce a rate for the estimation of $I$.

\subsection{The joint index of dependence}

The natural plug-in estimator for $I_{\otimes}$ requires a sample from $\mathscr{L}(X) \otimes \mathscr{L}(Y)$ as well. One possibility is to use 1/3 of the sample to estimate the joint distribution and the remaining 2/3 to estimate the product distribution. Under the simplifying assumption that $n$ is divisible by 3,
\begin{align*}
&\hat P_{{\mathscr{L}}(X,Y)} = \frac{3}{n} \sum_{i=1}^{n/3} \delta_{(X_i,Y_i)}; \quad \hat P_{\otimes} = \frac{3}{n} \sum_{i=1}^{n/3} \delta_{\big(X_{\frac{n}{3}+i},Y_{\frac{2n}{3}+i}\big)},
\end{align*}
where $\hat P_{\otimes} = \hat P_{\mathscr{L}(X) \otimes \mathscr{L}(Y)}$.
The plug-in estimator above has the advantage of using independent samples for the joint distribution and the independent coupling, thus reducing the analysis to the standard two-sample plug-in estimator of the Wasserstein distance. 
The speed of convergence in the Glivenko-Cantelli theorem in Wasserstein distance is well understood \citep{Fournier2015rate}, and in principle can be used to obtain convergence rates for the two-sample problem as in our case through triangle inequality. However, recent works have obtained faster rates leveraging smoothness of the cost, and, when applicable, separation of the measures \citep{chizat2020faster,ManoleWeed2024,Hundrieser2024}; we adapt them to our setting in the following theorem.

The first part can be found in \citet[Corollary 3]{ManoleWeed2024} for the case $2m \geq 5$, and in \citet[Sections 3.3 and 3.4]{Hundrieser2024} for the general case; 
the second part comes from the estimate $|a-b| \leq a^{1-p} |a^p - b^p|$ valid for any $a,b > 0$, as in \citet[Corollary 4]{ManoleWeed2024}.

\begin{theorem}
Let $(X_i,Y_i) \simiid \mathscr{L}(X,Y)$  on $[0,1]^{m} \times [0,1]^{m}$ for $i=1,\dots, n$. 
If $d$ is the Euclidean distance on $\mathbb{R}^{2m}$ and $p\geq 1$, then there exist constants $C(m,p)>0$ such that
\begin{equation*}
\E(|D_{\otimes}(X,Y; d,p)^p -\mathcal{W}_{d,p}(\hat P_{{\mathscr{L}}(X,Y)}, \hat P_{ \otimes})^p|)  
\le  C(m,p) r(n),
\end{equation*}
where the rate $r(n) = r_{m,p}(n)$ is defined as 
\begin{equation*}
r(n) = \begin{cases}
n^{-1/2} & \text{if } m < \min(p,2), \\
n^{-1/2} \log(n) & \text{if } m = \min(p,2), \\
n^{-\min(p,2)/(2m)} & \text{if } m > \min(p,2).
\end{cases}
\end{equation*}
Moreover, if $X$ and $Y$ are not independent,
\begin{equation*}
\E(|D_{\otimes}(X,Y; d,p) -\mathcal{W}_{d,p}(\hat P_{{\mathscr{L}}(X,Y)}, \hat P_{ \otimes})|) 
\le  \frac{C(m,p)}{\mathcal{W}_{d,p}(\mathscr{L}(X,Y),  \mathscr{L}(X)  \otimes \mathscr{L}(Y))^{p-1}}  \, r(n).
\end{equation*}
\end{theorem}

This result implies that the plug-in estimator recovers a parametric rate (up to a logarithmic factor) if $m \leq \min(p,2)$. While the result gives convergence in expectation, it can be shown that $\mathcal{W}_{d,p}(\hat P_{{\mathscr{L}}(X,Y)}, \hat P_{ \otimes})$ concentrates well around its expectation, in the sense that it is sub-Gaussian with sub-Gaussian constant scaling as $n^{-1/2}$ \citep[Theorem 2]{chizat2020faster}. 


If $n$ is moderately large using only $n/3$ samples to estimate the joint distribution can considerably increase the bias. At the price of introducing dependence between the samples for the joint and independent distributions, one can consider a finite permutation $\sigma \in \mathcal{S}_n$ and define
\begin{align*}
&\hat{P}_{\mathscr{L}(X,Y)}  = \frac{1}{n} \sum_{i=1}^{n} \delta_{(X_i,Y_i)}; \quad \hat{P}_{ \otimes } = \frac{1}{n} \sum_{i=1}^{n} \delta_{(X_{i},Y_{\sigma(i)})}. 
\end{align*}
\cite{XiaoWang2019} use this second method, borrowed from \cite{ArconesGine1992}, which is not specific to the Wasserstein distance. A third option is to estimate the product measure with $n^2$ dependent pairs $\hat{P}_{\mathscr{L}(X)} \otimes \hat{P}_{\mathscr{L}(Y)}$, which is equal to
\[
\frac{1}{n^2} \sum_{i,j=1}^{n} \delta_{(X_{i},Y_{j})} = \frac{1}{n} \sum_{i=1}^{n} \delta_{X_{i}} \otimes \frac{1}{n} \sum_{j=1}^{n} \delta_{Y_{j}}.
\]
This has been used by \cite{Liu2022} to estimate an entropy regularized version of the Wasserstein distance. As therein, we expect the analysis of the sample complexity to be more difficult to prove. 

The previous estimators provide natural ways to estimate $D_{\otimes}$ but extra care is needed to recover the sample complexity of the plug-in estimator of $I_{\otimes}$ because of the presence of the denominator. Interestingly, a thorough analysis of the sample complexity of $I_\otimes$ still seems to be missing in the literature.

\subsection{The conditional index of dependence}

The sample complexity of $I_{|\bullet}$ can also be studied through a plug-in estimator, as detailed in \cite{wiesel2022measuring} for $p=1$ and $d$ the standard Euclidean metric. Instead of creating an artificial sample for the independent coupling, as for $D_{\otimes}$, the delicate part is to estimate the conditional distribution $\mathscr{L}(Y|X)$. To guarantee that it can be done efficiently, the author introduces the assumption that a version of the conditional law $x \mapsto \mathscr{L}(Y|X=x)$ is $\W$-Lipschitz with respect to $x$, in the following sense: there exists $L \geq 0$ and an event $A$ with $\mathbb{P}(X \in A) = 1$ such that, for all $x, x' \in A$,
\begin{equation}
\label{def:wlipschitz}
\W(\mathscr{L}(Y|X=x), \mathscr{L}(Y|X=x')) \leq L d_\mathbb{X}(x,x'). 
\end{equation}

To define an estimator of $\mathscr{L}(Y|X)$, the natural stratgegy, see e.g. \cite{backhoff2022estimating}, it to first divide $[0,1]^m$ into $\Phi_n$ disjoint cubes and send the components of each observation pair to the center of the nearest cube through a map denoted by $\phi_n$. One defines
\begin{equation*}
\hat{P}_{\mathscr{L}(Y|X=x)}
=\frac{1}{|\{j:\phi_{n}(X_j) = \phi_{n}(x)\}|} \sum_{j:\phi_{n}(X_j) = \phi_{n}(x) } \delta_{\phi_{n}(Y_j)},
\end{equation*}
where $| \cdot |$ is the cardinality of a set. 

\begin{proposition}[Theorem 6.2 in \cite{wiesel2022measuring}]
Let $(X_i,Y_i) \simiid \mathscr{L}(X,Y)$  on $[0,1]^{m} \times [0,1]^{m}$ for $i=1,\dots, n$, 
such that a version of the map $x \mapsto \mathscr{L}(Y|X=x)$ is $\mathcal{W}$-Lipschitz with constant $L$. 
With $d$ the Euclidean distance, $p=1$ and choosing $\Phi_n = n^{1/3}$ if $m=1$ and $\Phi_n = n^{1/2}$ if $m \geq 2$,
\begin{multline*}
\E\bigg( \bigg|D_{| \bullet}(X,Y; d, 1) - \frac{1}{n} \sum_{i=1}^n \mathcal{W}_{d,1}\bigg(\hat{P}_{\mathscr{L}(Y|X=\phi_{n}(X_i))}, 
\frac{1}{n} \sum_{i=1}^n \delta_{\phi_{n}(Y_j)})\bigg)\bigg| \bigg) 
 \\ \le C(m,L,\mathscr{L}(Y))
\begin{cases}
n^{-1/3} & \text{if } m=1 \\
n^{-1/4}\log(n) & \text{if } m=2\\
n^{-1/(2m)} & \text{if } m \ge 3,
\end{cases}
\end{multline*}
where $C(m,L,\mathscr{L}(Y))$ is a constant that depends on the dimension $m$, the Lipschitz constant $L$, and $\mathscr{L}(Y)$. 
\end{proposition}

The proof is based on the rate of the empirical estimator of the adapted Wasserstein distance in \cite{backhoff2022estimating}. 
\cite{wiesel2022measuring} remarkably shows that the same rates hold for $I_{|\bullet}$, and proves a sub-Gaussian concentration bound of the estimator of $D_{|\bullet}$ around its expectation.

\subsection{Entropic regularization}

In both cases the rates exhibit the ``curse of dimensionality'', in the sense that the dependency of the rate is exponential in the dimension $m$.
This is actually a well-known problem of optimal transport \citep{dudley1969speed}, and it may limit the applicability of the indices above as a huge number of samples are required to estimate them with good precision.


A relatively recent proposal to circumvent this issue is the entropic regularization of optimal transport \citep{cuturi2013sinkhorn,Peyre2019computational,nutz2021introduction}. 
Given a cost function $c : \mathbb{M} \times \mathbb{M} \to [0, + \infty)$ and a regularization parameter $\varepsilon > 0$, the entropic Wasserstein discrepancy $\W^{(\varepsilon)}_{c}(P,Q)$ between $P,Q \in \mathcal{P}(\mathbb{M})$ is defined as 
\begin{equation*}
 \inf_{\gamma \in \Gamma(P,Q)} \left\{ \E_{(Z,Z') \sim \gamma}(c(Z,Z')) + \varepsilon \mathrm{KL}(\gamma | P \otimes Q) \right\}.  
\end{equation*}
Here $\mathrm{KL}(\gamma | P \otimes Q)$ is the Kullback-Leibler divergence on the product space $\mathbb{M} \times \mathbb{M}$, defined as 
\[
\begin{cases}
\E_{(Z,Z') \sim \gamma} \left( \log f(Z,Z') \right) & \text{if } \gamma \ll P \otimes Q \\
+\infty & \text{otherwise},
\end{cases}
\]
where $f$ is the density of $\gamma$ with respect to $P \otimes Q$. This forces the infimum over all couplings $\gamma \in \Gamma(P,Q)$ to be absolutely continuous with respect to $P \otimes Q$, which excludes the classical optimal coupling.
As with $\W$, here $c$ is also not necessarily a distance, and typically is taken to be the square of a distance.
As $\varepsilon \to 0$, $\W^{(\varepsilon)}_{c}$ converges to $\W_{c,1}$: in this sense $\W^{(\varepsilon)}_{c}$ can be thought of as an approximation to $\W_{c,1}$.  

In addition to the better computational properties that we comment below, 
an advantage of the entropic optimal transport problem is the sample complexity: if $\hat P_n$, $\hat Q_n$ 
are empirical measures obtained from $n$ i.i.d. samples of 
$P,Q \in \mathcal{P}(\R^m)$, then 
\begin{equation*}
\E \left( \left| \W^{(\varepsilon)}_{c}( \hat P_n, \hat Q_n) - \W^{(\varepsilon)}_{c}(P,Q) \right| \right) \leq \frac{C(\varepsilon,m,P,Q)}{\sqrt{n}},
\end{equation*}
for some constant $C(\varepsilon,m,P,Q)$ which depends on the sub-Gaussian constants of $P$ and $Q$, and blows up to $+ \infty$ as $\varepsilon \to 0$ \citep{genevay2019sample,mena2019statistical}.
In particular, for a given 
$\varepsilon > 0$, the rate is parametric in $n$, independently on the dimension $m$. 

Given 
the better sample complexity in high dimensions, it may be tempting to replace
the classical Wasserstein distance $\W$ by the entropic one $\W^{(\varepsilon)}$ in all the proposed indices. However, one has to be careful as some theoretical properties are lost. Most importantly, though $\W^{(\varepsilon)}_{c} \geq 0$, we have 
$\W^{(\varepsilon)}_{c}(P,P) > 0$ even if $c$ is non-degenerate. Thus the index built with $\W^{(\varepsilon)}_{c}$ would fail to detect independence. A surprisingly simple fix to retrieve non-degeneracy 
is to introduce the 
Sinkhorn divergence $S^{(\varepsilon)}_{c}(P,Q)$ defined as
\begin{equation*}
\W^{(\varepsilon)}_{c}(P,Q) - \frac{1}{2} \W^{(\varepsilon)}_{c}(P,P) - \frac{1}{2} \W^{(\varepsilon)}_{c}(Q,Q). 
\end{equation*}
By construction $S^{(\varepsilon)}_{c}(P,P) = 0$. 
As $\varepsilon \to 0$, $S^{(\varepsilon)}_{c}$ is still an approximation of $\mathcal{W}_{c,1}$, while as $\varepsilon \to + \infty$ it converges to a Maximum Mean Discrepancy distance with kernel $-c$, provided it is a semi-positive definite kernel \citep{feydy2019interpolating}. The latter have been used to measure dependence \citep{Gretton2005, deb2020measuring}. 
Moreover, 
$S^{(\varepsilon)}_{c}(P,Q) \geq 0$ with equality if and only if $P = Q$, under some assumptions that are satisfied if $c$ is the squared Euclidean cost on $\R^m$ \citep{feydy2019interpolating}. 
The sample complexity of $S^{(\varepsilon)}_{c}$ is as good as that of $\W^{(\varepsilon)}_{c}(P,Q)$. Based on these considerations, \cite{Liu2022} proposed 
\begin{equation*}
D^{(\varepsilon)}_{\otimes}(X,Y) = S^{(\varepsilon)}_{c}(\mathscr{L}(X,Y),\mathscr{L}(X) \otimes \mathscr{L}(Y)). 
\end{equation*}
By construction $D^{(\varepsilon)}_{\otimes}(X,Y) \geq 0$ with equality if and only if $X$ and $Y$ are independent. Moreover, if $(X_i,Y_i)_{i=1,\ldots,n}$ are i.i.d. copies of a sub-Gaussian pair $(X,Y)$ in $\R^m \times \R^m$,
with $\hat P_{\mathscr{L}(X,Y)} = n^{-1} \sum_{i=1}^n \delta_{(X_i,Y_i)}$ the empirical distribution,
\citet{Liu2022} show that
\begin{equation*}
\E \left( \left| D^{(\varepsilon)}_{\otimes}(\hat P_{\mathscr{L}(X,Y)}) - D^{(\varepsilon)}_{\otimes}(X,Y) \right| \right) 
\leq  C(m) \left( \varepsilon + \frac{\sigma^{\lceil 5m/2 \rceil + 6}  }{\varepsilon^{\lceil 5m/4 \rceil + 2}} \right) \cdot \frac{1}{\sqrt{n}}.
\end{equation*}
where $\sigma$ is the sub-Gaussian constant of $X$ and $Y$. The dependence is now parametric in $n$.
Though $D^{(\varepsilon)}_{\otimes}(X,Y)$ is an attractive candidate to measure proximity to independence, not much is known about tractable upper bounds in order to normalize it to build an index, with a neat characterization of extremal cases. 

Regarding the conditional index of dependence, \cite{borgonovo2024global} also consider
\begin{equation*}
D^{(\varepsilon)}_{|\bullet}(X,Y) = \E_X(\mathcal{W}^{(\varepsilon)}_{c}(\mathscr{L}(Y|X), \mathscr{L}(Y))).
\end{equation*}
As $\W^{(\varepsilon)}_{c}$ is not debiased, thus $\mathcal{W}^{(\varepsilon)}_{c}(P,P) > 0$, 
this index does not characterize independence: $D^{(\varepsilon)}_{|\bullet}(X,Y) > 0$ even if $X$ and $Y$ are independent. On the other hand, the authors show that the upper bound 
stays valid. Being $Y'$ an independent copy of $Y$,
\begin{equation*}
D^{(\varepsilon)}_{|\bullet}(X,Y) \leq \E(c(Y,Y')),
\end{equation*}
with equality if and only if $Y = f(X)$ a.s. for a measurable function $f$. 

\begin{remark}
Sliced Wasserstein distances are another way to circumvent the curse of dimensionality of classical optimal transport. In $\R^d$, with $\theta$ a random vector distributed uniformly over the unit sphere and independent of $X$ and $Y$, the sliced Wasserstein distance between $\mathscr{L}(X)$ and $\mathscr{L}(Y)$ is 
\begin{equation*}
\E_\theta \left( \mathcal{W}_p((\mathscr{L}(\theta \cdot X),\mathscr{L}(\theta \cdot Y))^p \right)^{1/p},
\end{equation*}
where $\cdot$ denotes the scalar product in $\R^d$. Introduced by \citet{rabin2011wasserstein}, the interest of this metric is twofold: first it involves only Wasserstein distances between the one-dimensional variables $\theta \cdot X$ and $\theta \cdot Y$, for which explicit formulas are available (see Section~\ref{sec:analytic}) and enable fast numerical evaluation; second, the sample complexity of the sliced Wasserstein is much better than the one of the classical Wasserstein distance \citep{nadjahi2020statistical,nietert2022statistical}.  \cite{li2023scalable} and \cite{zhao2021sliced} use the sliced Wasserstein distance to measure dependence. We expect several results of this survey to extend when classical Wasserstein distances are substituted by sliced Wasserstein distances.
\end{remark}

\section{Computational aspects}
\label{sec:comp_comp}

In the previous section, we saw how to build estimators of the indices, but these estimators still require to solve optimal transport problems between discrete measures. If $Z,Z'$ are discrete, in the sense that they take a finite number of values, then $\W(\mathscr{L}(Z),\mathscr{L}(Z'))$ can be computed by solving a linear program: we refer to \cite{Peyre2019computational,merigot2021optimal} and references therein for a thorough introduction to the topic. 

If $Z,Z'$ 
take at most $N$ values, 
then computing $\W(\mathscr{L}(Z),\mathscr{L}(Z'))$ exactly requires $O(N^3 \log(N))$ operations \citep{orlin1997polynomial}. There are 
alternative algorithms
if one is willing to make an approximation error, such as the auction algorithm \citep{bertsekas1979distributed}, which can be attractive since we usually use an approximation of our target distribution anyway (e.g., the corresponding empirical measure). However, these algorithms still have a complexity in $N$ that is at least quadratic. This complexity of the solvers, coupled with the statistical complexity described in the previous section, makes the overall computational complexity of estimating the index $I_\otimes$ or $I_{|\bullet}$ 
prohibitive in high dimensions. 

As explained in the previous section, a proposal to reduce the sample complexity is to use $\W_c^{(\varepsilon)}$ the entopic-regularized optimal transport distance. 
Its computation is possible via the so-called Sinkhorn algorithm. Although the computational complexity in this case also scales quadratically in $N$, with well crafted implementations, it can be computed in practice much faster than $\W$. 
This was also a motivation for the work \cite{Liu2022} mentioned earlier. The literature on the speed of convergence of Sinkhorn's algorithm is quite vast and we refer to \cite{Peyre2019computational,nutz2021introduction} and references therein for details and statements of various results.

As this is not the focus of this survey, we do not report 
a description of the algorithms. We advise the practitioners to use off the shelf libraries such as POT \citep{flamary2021pot} or OTT \citep{cuturi2022optimal}, which contain robust implementations of the algorithms evoked in this section. 

\section{An alternative road: distance from maximal dependence}
\label{sec:alternative}

The indices described so far aim at detecting independence, which is a well-defined coupling on any product measure space $\X \times \Y$. Maximal dependence can be defined as those couplings whose index is equal to 1 and, as we have seen in Section~\ref{sec:index}, this notion can vary according to the discrepancy $D\in \{D_{ \otimes}, D_{|\bullet}\}$ and the ground cost $d$. The most popular notions of maximal dependence can be formulated in terms of a class of degenerate couplings.

\begin{definition}
Let $\mathcal{I}_1 \subset \{f:\mathbb{X} \to \mathbb{Y} \text{ Measurable}\}$. Then the set of $\mathcal{I}_1$-maximally dependent couplings is $\{(X,f(X)): f \in \mathcal{I}_1\}$.
\end{definition}

Common choices of
$\mathcal{I}_1$ include measurable functions \citep{wiesel2022measuring, nies2025transport,borgonovo2024global}, 
Lipschitz functions \citep{nies2025transport}, 
similarities \citep{mori2020earth, nies2025transport}, 
optimal transport maps (conjectured for \cite{mordant2022measuring}), linear functions \citep{Puccetti2022}, or even the identity function when $\X = \Y$. 
As nicely underlined in \cite{nies2025transport}, the natural notion of maximal dependence can vary depending on the context, and should be carefully taken into account when choosing an index of dependence. In some situations, the focus is on detecting the \emph{intensity of the statistical relationship} \citep{CifarelliRegazzini2017} rather than independence. Historically, for random variables in $\R$, this need has led to the definition of measures of concordance and Gini's indices of homophily; we refer to \cite{CifarelliRegazzini2017} for an accurate review.
With this point of view in mind, we argue that optimal transport could also be used to compute the distance from 
$\mathcal{I}_1$-maximal dependence as
\begin{equation}
\label{def:maximal}
D_{\mathcal{I}_1}(X,Y) = \inf_{f \in \mathcal{I}_1}\mathcal{W}(\mathscr{L}(X,Y), \mathscr{L}(X,f(X))).
\end{equation}
If we can find a minimizer 
$f^* \in \mathcal{I}_1$ for this problem, 
$D_{\mathcal{I}_1}$ could be a compelling alternative because it involves computing a Wasserstein distance with respect to a distribution $\mathscr{L}(X,f^*(X))$ with a degenerate support. This could sensibly reduce the sampling complexity of the evaluation of the Wasserstein distance, since it is known to adapt to the smallest dimension of the supports of the two distributions \citep{Hundrieser2024}. 

To the best of our knowledge, there are no works addressing this framework for general Polish spaces. However, related results have been established under stronger structural assumptions. When $\X = \Y = \R$ and $\mathcal{I}_1 = \{{\rm id}\}$ contains only the identity map ${\rm id}(x) = x$, the following result shows that the Wasserstein distance can be evaluated explicitly, in the spirit of Section~\ref{sec:analytic}.

\begin{theorem}[{Theorem 20 in \citet{Catalano2021}}]
\label{th:aos}
Let $(X,Y)$ such that $\mathscr{L}(X)$ and $\mathscr{L}(Y)$ are absolutely continuous with respect to the Lebesgue measure on $\mathbb{R}$. Then, 
\[
(x,y) \mapsto (F_X^{-1} \circ F_{X+Y}(x+y ), F_X^{-1} \circ F_{X+Y}(x+y)),
\]
is an optimal transport map 
between $\mathscr{L}(X, Y)$ and $\mathscr{L}(X, X)$ for $\mathcal{W}_{d,2}$, where $d$ is the Euclidean distance. In particular $\W_{d,2}(\mathscr{L}(X,Y), \mathscr{L}(X,X))^2$ is equal to
\begin{equation*}
\E((X-F_X^{-1}(F_{X+Y}(X+Y)))^2) 
+ \E((Y-F_X^{-1}(F_{X+Y}(X+Y)))^2).
\end{equation*}
\end{theorem}

This result is related to the ``multi-to-one'' optimal transport \citep{chiappori2017multi}, as a two-dimensional 
distribution is transported to a one-dimensional density.

Starting from the representation for random variables in $\mathbb{R}$, \cite{Catalano2021} adapted this measure to quantify dependence between completely random measures \citep{Kingman1967} on a Polish space $\X$, the key idea being that for any random measure $\tilde \mu: \Omega \to \mathcal{M}(\mathbb{X})$, the set-wise evaluations $\tilde \mu(A)$ are random variables on $[0,+\infty)$. \cite{BeneventoDurante2024} used a similar idea on copulas to cluster time series.  Solutions for more general classes of functions 
$\mathcal{I}_1$ remain unexplored.

From the discrepancy 
$D_{\mathcal{I}_1}$ in principle one could define an index of dependence following the route of Section~\ref{sec:principled}, finding an upper bound of the form
\begin{equation}
\label{sup_maximal}
U_{\mathcal{I}_1}(X,Y) = \sup_{\mathscr{L}(X',Y') \in \Gamma(X,Y)}  D_{\mathcal{I}_1}(X',Y'). 
\end{equation}
When 
$\mathcal{I}_1 = \{\rm{id}\}$ this amounts to finding the coupling that is at a maximal distance from the ``diagonal'' coupling 
change{$\mathscr{L}(X,X)$}. 
It is a difficult problem, though possibly not as difficult as finding the maximal distance from independence thanks to the expression of the optimal coupling in Theorem~\ref{th:aos}. 
We present here a new result:
a solution of \eqref{sup_maximal} under some symmetry assumption.
The techniques are inspired by \citet{Catalano2024}, which we describe later in this section.

\begin{theorem} 
\label{th:max}
Let $X$ an absolutely continuous random variable on $\R$ with finite second moment, satisfying the symmetry condition $(a-X) \eqd X$ for some $a \in \R$. 
Let $\gamma_{+} = \mathscr{L}(X,X)$, $\gamma_{-} = \mathscr{L}(X, a-X)$, and  $\W = \W_{d,2}$ where $d$ is the Euclidean distance on $\R^2$. Then,
\begin{equation*}
\sup_{\gamma \in \Gamma(X,X)} \W(\gamma,\gamma_+) = \W(\gamma_-,\gamma_+).
\end{equation*}
\end{theorem}

\begin{proof}
Let $\gamma = \mathscr{L}(X,Y) \in \Gamma(X,X)$. The transport map of Theorem~\ref{th:aos} between $\gamma$ and $\gamma_+$ can be written as the gradient of the convex function $(x,y) \mapsto  u(x+y)$, where $u : \R \to \R$ is a convex function satisfying
$u'(z) = F_X^{-1} (F_{X+Y}(z))$. Thus we can find a pair $(\varphi,\psi)$ of optimal Kantorovich potentials for the problem of sending $\gamma$ onto $\gamma_+$ with  $\varphi(x,y) = (|x|^2 + |y|^2)/2 - u(x+y)$ \citep[Section 1.3.1]{Santambrogio2015}. By Kantorovich duality 
\citep[Theorem 1.40]{Santambrogio2015},
\begin{align*}
\W^2(\gamma,  \gamma_+) & = \E(\varphi(X,Y)) + \E(\psi(X,X)), \\
\W^2(\gamma_-,  \gamma_+) & \geq \E(\varphi(X, a-X)) + \E(\psi(X,X)).
\end{align*}
Subtracting these expressions, given the expression of $\varphi$ and as $X \eqd Y \eqd ( a-X)$,  we obtain
\begin{align*}
\W^2(\gamma_-,  \gamma_+)  - \W^2(\gamma,\gamma_+) & \geq \E\left( u(X+Y) \right) - \E \left( u(X + a-X) \right) \\
& = \E(u(X+Y)) - u(a). 
\end{align*}
As $u$ is convex and $\E(X+Y) = a$, by Jensen's inequality the last term is non-negative. Thus $\W(\gamma,\gamma_+) \leq \W(\gamma_-,\gamma_+)$, which is the conclusion as $\gamma$ is arbitrary.  
\end{proof}




Theorem~\ref{th:max} determines the expression for a new index of concordance,   
\[ 
I(X,Y) = 1- 2 \frac{\W(\mathscr{L}(X,Y),\mathscr{L}(X,X))}{\W(\mathscr{L}(X, a-X),\mathscr{L}(X,X))}, 
\]
which takes values in $[-1,1]$ and distinguishes between positive and negative dependence. Since the uniform distribution on [0,1] is absolutely continuous and satisfies the symmetry condition with $a = 1$, the most natural use of the index appears to be on copulas, which also guarantees invariance with respect to bijective increasing tranformations. We leave to future work a detailed exploration of the properties of this index.


A solution to 
problem \eqref{sup_maximal} for a slightly different setting has been found in \cite{Catalano2024}, where instead of couplings 
$\mathscr{L}(X,Y)$ between probabilities the focus is on measures with infinite mass and bounded second moments. These conditions are satisfied by many of the most common multivariate L\'evy measures appearing in the theory of infinite divisibility for dependent L\'evy processes (see e.g., \cite{Sato1999,ContTankov2004}) and dependent completely random measures \citep{Kingman1967}. The infinite mass of the L\'evy measures requires an extension of the notion of Wasserstein distance, introduced by \cite{FigalliGigli2010} under slightly different assumptions and formalized by \cite{Guillen2019} on L\'evy measures. 

Let $\Omega_m=[0,+\infty)^m \setminus \{0\}$ and let $\mathcal{M}_2 (\Omega_m)$ denote the set of positive Borel measures $\nu$ on $\Omega_m$ with finite second moment $M_2(\nu)=\int_{\Omega_m}\|s\|^2 \mathrm{d} \nu(s)<+\infty$. 

Let $\pi_i(s_1,s_2) = s_i$ be the $i$-th projection for $i=1,2$ and let $\#$ denote the pushforward of a measure, i.e. $f_\#\nu(A) =  \nu(f^{-1}(A))$ for every measurable $f$ and Borel set $A$. For $\gamma \in \mathcal{M}_2\left(\Omega_{2 m}\right)$, $(\pi_{i})_{\#} \gamma$ are measures on $[0,+\infty)^m$. We denote by $(\pi_{i} )_{\#} \gamma |_{\Omega_m}$ their restrictions to $\Omega_m$.



\begin{definition}
The set $\bar{\Gamma}(\nu^1, \nu^2)$ of extended couplings between $\nu^1,\nu^2 \in \mathcal{M}_2 (\Omega_m)$ is
\begin{equation*}
\{\gamma \in \mathcal{M}_2(\Omega_{2 m}):\, (\pi_{i})_{\#} \gamma |_{\Omega_m}=\nu^i \text{ for }i=1,2 \}.
\end{equation*}
\end{definition}

Extended couplings are compact with respect to the weak$^*$ topology on $\sigma$-finite measures. This makes it possible to define the corresponding Wasserstein distance.

\begin{definition}
The extended Wasserstein distance between $\nu^1, \nu^2 \in \mathcal{M}_2(\Omega_m)$ is
\[
\mathcal{W}_*(\nu^1, \nu^2)^2=\inf _{\gamma \in \bar{\Gamma}(\nu^1, \nu^2)} \iint_{\Omega_{2 m}} \|s_1-s_2 \|^2 \mathrm{d} \gamma (s_1, s_2).
\]
\end{definition}

There are two extended couplings that play a fundamental role in the study of dependence: the ``diagonal'' extended coupling $\gamma^{\rm co} = ({\rm id}, {\rm id})_{\#} \nu$, and the ``axis'' extended coupling $\gamma^{\perp} = (0,{\rm id})_{\#} \nu + ({\rm id},0)_{\#} \nu$, for $\nu \in \mathcal{M}_2(\Omega_m)$. When using multivariate L\'evy measures to  model dependent L\'evy processes or dependent completely random measures, $\gamma^{\rm co}$ corresponds to maximal dependence, whereas $\gamma^{\perp}$ corresponds to independence.

\begin{theorem}[{Theorem 1 in \cite{Catalano2024}}]
If $\nu \in \mathcal{M}_2(\Omega_1)$ then 
\[
\sup_{\gamma \in \bar{\Gamma}(\nu, \nu)} \mathcal{W}_*(\gamma, \gamma^{\rm co}) = \mathcal{W}_*(\gamma^{\perp}, \gamma^{\rm co})
\]
\end{theorem}

In other terms, at least under the assumption of equal marginal distributions, the maximal distance from maximal dependence is retrieved under independence. This allows to define an index of dependence in [0,1] that detects both complete dependence and independence, following the desiderata of Section~\ref{sec:principled}. Moreover, since we are computing a distance with respect to a coupling with degenerate support, we are able to retrieve exact expressions for the distance, in the same spirit of Theorem~\ref{th:aos}.

It is still an open direction to understand whether the techniques of \cite{Catalano2024} could be employed to solve \eqref{sup_maximal} for more general classes of functions $\mathcal{I}_1$.

\section{Further proposals}
\label{sec:further}



\subsection{Conditional independence}



The measures of dependence we discussed can be adapted to measures of conditional independence with respect to a third variable $Z$ defined on the same probability space. One can consider
\begin{enumerate}
\item $\E_Z(\mathcal{W}(\mathscr{L}(X,Y|Z),  \mathscr{L}(X|Z)  \otimes \mathscr{L}(Y|Z)))$;
\item $\E_{X,Z}( \mathcal{W}(\mathscr{L}(Y|X,Z), \mathscr{L}(Y|Z)))$.
\end{enumerate}
Both quantities are non-negative, and equal to zero if and only if $(X,Y)$, conditionally to $Z$, are independent.
In particular, 1. has been used in the context of conditional independence testing in \cite{warren2021wasserstein, Neykov2024}. To estimate 1. through samples, both works assume that the map $z \mapsto \mathscr{L}(X,Y|Z=z)$ is $\mathcal{W}$-Lipschitz, which is a similar assumption to the one for the estimator of the conditional index of dependence; see \eqref{def:wlipschitz}.

\subsection{More than two random variables}
\label{sec:more_than_2}

Instead of two random variables, let us consider a measure of dependence for $m$ random variables $X_1, \ldots, X_m$, for simplicity defined on the same space $\mathbb{X}$. It is straightforward to extend the joint index of dependence: with $d$ distance on $\mathbb{X}^m$ we can define $D_\otimes(X_1, \ldots,X_m)$ as 
\begin{equation*}
\W_{d,p} \left( \mathscr{L}(X_1, \ldots, X_m), \mathscr{L}(X_1) \otimes \ldots \otimes \mathscr{L}(X_m) \right).
\end{equation*}
Then $D_\otimes \geq 0$, and $D_\otimes(X_1, \ldots,X_m) = 0$ if and only if $X_1, \ldots,X_m$ are (mutually) independent. As in the case $m=2$, for a generic $m \ge 2$ the measure is symmetric if the cost function is symmetric, in the sense $D_\otimes(X_1, \ldots,X_m) = D_\otimes(X_{\sigma(1)}, \ldots,X_{\sigma(m)})$ for any $\sigma$ permutation of $\{1,\ldots,m \}$. It is interesting to understand the relation between measures of dependence across dimensions, with an immediate consequence being the following.

\begin{lemma}
Assume $(X_1, \ldots, X_m)$ is independent from $X_{m+1}$ . Then, 
\begin{equation*}
D_\otimes(X_1, \ldots, X_{m+1}) \leq D_\otimes(X_1, \ldots,X_m). 
\end{equation*}
\end{lemma}

Finding appropriate upper bounds for building an index becomes more challenging, with many possible notions of maximal dependence. \cite{DeKeyserGijbels2025} extends the index \eqref{def:reference_measure} of \cite{mordant2022measuring} to this setting, providing a list of desirable properties. \cite{Catalano2024} uses a similar idea to measure the dependence for an arbitrary number of L\'evy measures with the same marginal distribution, computing a distance with respect to maximal dependence as in Section~\ref{sec:alternative}.

\subsection{Geometry of optimal transport}

The theory of optimal transport may be used to measure dependence in other ways. An interesting direction has been initiated by \cite{PetersenMueller2019} to measure dependence between random density functions using the geometry of optimal transport on $\R$.
The full description of the construction involves the tangent space to the space of probability measures, parallel transport, and Wasserstein barycenters, which would not fit in this survey. We only report the final expression: if $\tilde f_1$ and $\tilde f_2$ are two random densities, the resulting measure takes a simple form
\[
\cov_{\otimes}(\tilde f_1, \tilde f_2) = \int_0^1 \cov(\tilde F_1^{-1}(t), \tilde F_2^{-1}(t)) \mathrm{d}t, 
\]
where $\tilde F_i$ denote the corresponding random cumulative distribution functions, for $i=1,2$, and $\cov$ is the standard covariance between random variables in $\R$. 
It provides an alternative way to define a symmetric measure of dependence for a class of infinite-dimensional random objects, and it can be naturally normalized to an index in $[-1,1]$ by dividing it by $\cov_{\otimes}(\tilde f_1, \tilde f_1)^{1/2} \cov_{\otimes}(\tilde f_2, \tilde f_2)^{1/2}$, as proposed in \cite{mascari2026hilbert}. \citet{zhou2021intrinsic} extends this measure to stochastic processes taking values in the space of density functions.

\subsection{User prescribed notions of minimal and maximal dependence}

Another possibility is to choose directly the maximal and minimal dependence, and then build an index that detects the extremal cases, as described in \cite{Marti2017}. For given marginal laws $\mathscr{L}(X)$ and $\mathscr{L}(Y)$, let $C_1(X,Y)$ and $C_0(X,Y)$ denote two sets of couplings prescribed by the user, which will correspond to maximal and minimal dependence respectively. 
For example, $C_0(X,Y)$ could consist of the independent coupling, and $C_1(X,Y)$ of the couplings $(X,h(X))$ for $h$ belonging to $\mathcal{I}_1$ a prescribed class of functions, see Section~\ref{sec:alternative}.
Importantly, we may allow $C_0(X,Y)$ or $C_1(X,Y)$ to contain more than one element. With
\begin{equation*}
\W(\mathscr{L}(X,Y), C_i) = \min_{\gamma \in C_i(X,Y)} \W(\mathscr{L}(X,Y),\gamma),
\end{equation*}
for $i=0,1$, we define $I(X,Y)$ as
\begin{equation*}
\frac{\W(\mathscr{L}(X,Y), C_0)}{\W(\mathscr{L}(X,Y), C_0) + \W(\mathscr{L}(X,Y), C_1)}.
\end{equation*}
By construction $I(X,Y) \in [0,1]$, with $I(X,Y) = i$ if and only if $(X,Y) \in C_i(X,Y)$ for $i=0,1$.
Thus the sets $C_0(X,Y)$ and $C_1(X,Y)$ are user-prescribed, rather than imposed by the choice of an upper bound $U(X,Y)$. In principle, this strategy can be extended to measure different dependence assumptions, e.g.  closeness to a Gaussian copula for financial time series \citep{Marti2017}.

\section{Conclusions}
\label{sec:conclusions}

The rise of Wasserstein-based measures of dependence is driven by both the maturity of the field of Optimal Transport and the need to extend dependence concepts beyond real-valued random variables, unlocking new possibilities for analyzing dependence in more abstract spaces.

Despite many recent proposals, the field remains open. Even in the real-valued case, measuring the deviation from maximal dependence introduces new challenges and questions what a good notion of maximal dependence should be. More broadly, key questions remain on computational feasibility, robustness, and natural extensions to general metric spaces. Addressing these will further solidify the use of Wasserstein-based measures of dependence as a powerful tool in modern Statistics.

\section*{Acknowledgments}

The authors would like to thank Giacomo Nies, the Associate Editor and two anonymous referees for valuable comments that improved the manuscript.

\bibliography{bibsurvey}

\begin{thebibliography}{78}
\providecommand{\natexlab}[1]{#1}
\providecommand{\url}[1]{\texttt{#1}}
\expandafter\ifx\csname urlstyle\endcsname\relax
  \providecommand{\doi}[1]{doi: #1}\else
  \providecommand{\doi}{doi: \begingroup \urlstyle{rm}\Url}\fi

\bibitem[Aldous(1981)]{aldous}
D.~Aldous.
\newblock \emph{Weak convergence and the general theory of processes}, 1981.

\bibitem[Arcones and Gine(1992)]{ArconesGine1992}
M.~A. Arcones and E.~Gine.
\newblock {On the Bootstrap of $U$ and $V$ Statistics}.
\newblock \emph{The Annals of Statistics}, 20\penalty0 (2):\penalty0 655 --
  674, 1992.

\bibitem[Backhoff et~al.(2022)Backhoff, Bartl, Beiglb{\"o}ck, and
  Wiesel]{backhoff2022estimating}
J.~Backhoff, D.~Bartl, M.~Beiglb{\"o}ck, and J.~Wiesel.
\newblock Estimating processes in adapted wasserstein distance.
\newblock \emph{The Annals of Applied Probability}, 32\penalty0 (1):\penalty0
  529--550, 2022.

\bibitem[Backhoff-Veraguas et~al.(2020)Backhoff-Veraguas, Bartl, Beiglb{\"o}ck,
  and Eder]{backhoff2020all}
J.~Backhoff-Veraguas, D.~Bartl, M.~Beiglb{\"o}ck, and M.~Eder.
\newblock All adapted topologies are equal.
\newblock \emph{Probability Theory and Related Fields}, 178\penalty0
  (3):\penalty0 1125--1172, 2020.

\bibitem[Benevento and Durante(2024)]{BeneventoDurante2024}
A.~Benevento and F.~Durante.
\newblock Wasserstein dissimilarity for copula-based clustering of time series
  with spatial information.
\newblock \emph{Mathematics}, 12\penalty0 (1), 2024.
\newblock ISSN 2227-7390.

\bibitem[Bertsekas(1979)]{bertsekas1979distributed}
D.~P. Bertsekas.
\newblock A distributed algorithm for the assignment problem.
\newblock \emph{Lab. for Information and Decision Systems Working Paper, MIT},
  1979.

\bibitem[Borgonovo et~al.(2024)Borgonovo, Figalli, Plischke, and
  Savar{\'e}]{borgonovo2024global}
E.~Borgonovo, A.~Figalli, E.~Plischke, and G.~Savar{\'e}.
\newblock Global sensitivity analysis via optimal transport.
\newblock \emph{Management Science}, 2024.

\bibitem[Catalano et~al.(2021)Catalano, Lijoi, and Pr\"unster]{Catalano2021}
M.~Catalano, A.~Lijoi, and I.~Pr\"unster.
\newblock Measuring dependence in the {W}asserstein distance for {B}ayesian
  nonparametric models.
\newblock \emph{The Annals of Statistics}, 49\penalty0 (5):\penalty0
  2916--2947, 2021.

\bibitem[Catalano et~al.(2024)Catalano, Lavenant, Lijoi, and
  Prünster]{Catalano2024}
M.~Catalano, H.~Lavenant, A.~Lijoi, and I.~Prünster.
\newblock A wasserstein index of dependence for random measures.
\newblock \emph{Journal of the American Statistical Association}, 119\penalty0
  (547):\penalty0 2396--2406, 2024.

\bibitem[Chatterjee(2021)]{chatterjee2021new}
S.~Chatterjee.
\newblock A new coefficient of correlation.
\newblock \emph{Journal of the American Statistical Association}, 116\penalty0
  (536):\penalty0 2009--2022, 2021.

\bibitem[Chernozhukov et~al.(2017)Chernozhukov, Galichon, Hallin, and
  Henry]{chernozhukov2017monge}
V.~Chernozhukov, A.~Galichon, M.~Hallin, and M.~Henry.
\newblock Monge-{K}antorovich depth, quantiles, ranks and signs.
\newblock \emph{The Annals of Statistics}, 45\penalty0 (1):\penalty0 223--256,
  2017.

\bibitem[Chewi et~al.(2025)Chewi, Niles-Weed, and
  Rigollet]{chewi2024statistical}
S.~Chewi, J.~Niles-Weed, and P.~Rigollet.
\newblock \emph{Statistical optimal transport}.
\newblock Springer Cham, 2025.

\bibitem[Chiappori et~al.(2017)Chiappori, McCann, and Pass]{chiappori2017multi}
P.-A. Chiappori, R.~J. McCann, and B.~Pass.
\newblock Multi-to one-dimensional optimal transport.
\newblock \emph{Communications on Pure and Applied Mathematics}, 70\penalty0
  (12):\penalty0 2405--2444, 2017.

\bibitem[Chizat et~al.(2020)Chizat, Roussillon, L{\'e}ger, Vialard, and
  Peyr{\'e}]{chizat2020faster}
L.~Chizat, P.~Roussillon, F.~L{\'e}ger, F.-X. Vialard, and G.~Peyr{\'e}.
\newblock {Faster Wasserstein distance estimation with the Sinkhorn
  divergence}.
\newblock \emph{Advances in Neural Information Processing Systems},
  33:\penalty0 2257--2269, 2020.

\bibitem[Cifarelli and Regazzini(2017)]{CifarelliRegazzini2017}
D.~Cifarelli and E.~Regazzini.
\newblock On the centennial anniversary of {G}ini's theory of statistical
  relations.
\newblock \emph{METRON}, 75:\penalty0 227–242, 2017.

\bibitem[Cont and Tankov(2004)]{ContTankov2004}
R.~Cont and P.~Tankov.
\newblock \emph{Financial Modeling with Jump Processes}.
\newblock Chapman \& Hall/{CRC}, 2004.

\bibitem[Cuturi(2013)]{cuturi2013sinkhorn}
M.~Cuturi.
\newblock {Sinkhorn distances: Lightspeed computation of optimal transport}.
\newblock \emph{Advances in neural information processing systems}, 26, 2013.

\bibitem[Cuturi et~al.(2022)Cuturi, Meng-Papaxanthos, Tian, Bunne, Davis, and
  Teboul]{cuturi2022optimal}
M.~Cuturi, L.~Meng-Papaxanthos, Y.~Tian, C.~Bunne, G.~Davis, and O.~Teboul.
\newblock {Optimal Transport Tools (OTT): A JAX Toolbox for all things
  Wasserstein}.
\newblock \emph{arXiv preprint arXiv:2201.12324}, 2022.

\bibitem[{De Keyser} and Gijbels(2025)]{DeKeyserGijbels2025}
S.~{De Keyser} and I.~Gijbels.
\newblock High-dimensional copula-based wasserstein dependence.
\newblock \emph{Computational Statistics \& Data Analysis}, 204:\penalty0
  108096, 2025.
\newblock ISSN 0167-9473.

\bibitem[Deb et~al.(2020)Deb, Ghosal, and Sen]{deb2020measuring}
N.~Deb, P.~Ghosal, and B.~Sen.
\newblock {Measuring association on topological spaces using kernels and
  geometric graphs}.
\newblock \emph{arXiv preprint arXiv:2010.01768}, 2020.

\bibitem[Deb et~al.(2024)Deb, Ghosal, and Sen]{deb2024distribution}
N.~Deb, P.~Ghosal, and B.~Sen.
\newblock Distribution-free measures of association based on optimal transport.
\newblock \emph{arXiv preprint arXiv:2411.13080}, 2024.

\bibitem[Dette et~al.(2013)Dette, Siburg, and Stoimenov]{dette2013copula}
H.~Dette, K.~F. Siburg, and P.~A. Stoimenov.
\newblock A copula-based non-parametric measure of regression dependence.
\newblock \emph{Scandinavian Journal of Statistics}, 40\penalty0 (1):\penalty0
  21--41, 2013.
\newblock ISSN 03036898, 14679469.
\newblock URL \url{http://www.jstor.org/stable/23357251}.

\bibitem[Dowson and Landau(1982)]{dowson1982frechet}
D.~C. Dowson and B.~V. Landau.
\newblock {The Fréchet distance between multivariate normal distributions}.
\newblock \emph{Journal of multivariate analysis}, 12\penalty0 (3):\penalty0
  450--455, 1982.

\bibitem[Dudley(1969)]{dudley1969speed}
R.~M. Dudley.
\newblock {The speed of mean Glivenko-Cantelli convergence}.
\newblock \emph{The Annals of Mathematical Statistics}, 40\penalty0
  (1):\penalty0 40--50, 1969.

\bibitem[Eikenberry et~al.(2025)Eikenberry, Liu, and
  Lee]{eikenberry2025invariant}
K.~Eikenberry, L.~Liu, and Y.~Lee.
\newblock Invariant representations via wasserstein correlation maximization,
  2025.
\newblock URL \url{https://arxiv.org/abs/2505.11702}.

\bibitem[Feydy et~al.(2019)Feydy, S{\'e}journ{\'e}, Vialard, Amari, Trouv{\'e},
  and Peyr{\'e}]{feydy2019interpolating}
J.~Feydy, T.~S{\'e}journ{\'e}, F.-X. Vialard, S.-i. Amari, A.~Trouv{\'e}, and
  G.~Peyr{\'e}.
\newblock {Interpolating between optimal transport and MMD using Sinkhorn
  divergences}.
\newblock In \emph{The 22nd international conference on artificial intelligence
  and statistics}, pages 2681--2690. PMLR, 2019.

\bibitem[Figalli and Gigli(2010)]{FigalliGigli2010}
A.~Figalli and N.~Gigli.
\newblock A new transportation distance between non-negative measures, with
  applications to gradients flows with {D}irichlet boundary conditions.
\newblock \emph{Journal de Mathématiques Pures et Appliquées}, 94\penalty0
  (2):\penalty0 107--130, 2010.

\bibitem[Flamary et~al.(2021)Flamary, Courty, Gramfort, Alaya, Boisbunon,
  Chambon, Chapel, Corenflos, Fatras, Fournier, et~al.]{flamary2021pot}
R.~Flamary, N.~Courty, A.~Gramfort, M.~Z. Alaya, A.~Boisbunon, S.~Chambon,
  L.~Chapel, A.~Corenflos, K.~Fatras, N.~Fournier, et~al.
\newblock {Pot: Python optimal transport}.
\newblock \emph{Journal of Machine Learning Research}, 22\penalty0
  (78):\penalty0 1--8, 2021.

\bibitem[Fournier and Guillin(2015)]{Fournier2015rate}
N.~Fournier and A.~Guillin.
\newblock {On the rate of convergence in Wasserstein distance of the empirical
  measure}.
\newblock \emph{Probability Theory and Related Fields}, 162\penalty0
  (3):\penalty0 707--738, 2015.

\bibitem[Gelbrich(1990)]{gelbrich1990formula}
M.~Gelbrich.
\newblock {On a formula for the $L^2$ Wasserstein metric between measures on
  Euclidean and Hilbert spaces}.
\newblock \emph{Mathematische Nachrichten}, 147\penalty0 (1):\penalty0
  185--203, 1990.

\bibitem[Genevay et~al.(2019)Genevay, Chizat, Bach, Cuturi, and
  Peyr{\'e}]{genevay2019sample}
A.~Genevay, L.~Chizat, F.~Bach, M.~Cuturi, and G.~Peyr{\'e}.
\newblock Sample complexity of {S}inkhorn divergences.
\newblock In \emph{The 22nd International Conference on Artificial Intelligence
  and Statistics}, pages 1574--1583. PMLR, 2019.

\bibitem[Ghosal and Sen(2022)]{ghosal2022multivariate}
P.~Ghosal and B.~Sen.
\newblock {Multivariate ranks and quantiles using optimal transport:
  Consistency, rates and nonparametric testing}.
\newblock \emph{The Annals of Statistics}, 50\penalty0 (2):\penalty0
  1012--1037, 2022.

\bibitem[Gini(1914)]{Gini1914}
C.~Gini.
\newblock Di una misura della dissomiglianza tra due gruppi di quantità e
  delle sue applicazioni allo studio delle relazioni statistiche.
\newblock \emph{Atti del Reale Istituto Veneto di di Scienze, Lettere ed Arti},
  74:\penalty0 185--213, 1914.

\bibitem[Gretton et~al.(2005)Gretton, Bousquet, Smola, and
  Sch\"{o}lkopf]{Gretton2005}
A.~Gretton, O.~Bousquet, A.~Smola, and B.~Sch\"{o}lkopf.
\newblock Measuring statistical dependence with {H}ilbert-{S}chmidt norms.
\newblock In \emph{Proceedings of the 16th International Conference on
  Algorithmic Learning Theory}, ALT'05, page 63–77, Berlin, Heidelberg, 2005.
  Springer-Verlag.
\newblock ISBN 354029242X.

\bibitem[Guillen et~al.(2019)Guillen, Mou, and {\'S}wi{\c e}ch]{Guillen2019}
N.~Guillen, C.~Mou, and A.~{\'S}wi{\c e}ch.
\newblock Coupling {L}{\'e}vy measures and comparison principles for viscosity
  solutions.
\newblock \emph{Transactions of the American Mathematical Society},
  372\penalty0 (10):\penalty0 7327--70, 2019.

\bibitem[Hellwig(1996)]{hellwig1996sequential}
M.~F. Hellwig.
\newblock Sequential decisions under uncertainty and the maximum theorem.
\newblock \emph{Journal of Mathematical Economics}, 25\penalty0 (4):\penalty0
  443--464, 1996.

\bibitem[Hundrieser et~al.(2024)Hundrieser, Staudt, and Munk]{Hundrieser2024}
S.~Hundrieser, T.~Staudt, and A.~Munk.
\newblock {Empirical optimal transport between different measures adapts to
  lower complexity}.
\newblock \emph{Annales de l'Institut Henri Poincaré, Probabilités et
  Statistiques}, 60\penalty0 (2):\penalty0 824 -- 846, 2024.

\bibitem[Kendall(1938)]{Kendall1938}
M.~G. Kendall.
\newblock A new measure of rank correlation.
\newblock \emph{Biometrika}, 30\penalty0 (1/2):\penalty0 81--93, 1938.
\newblock ISSN 00063444.

\bibitem[Kingman(1967)]{Kingman1967}
J.~F.~C. Kingman.
\newblock Completely random measures.
\newblock \emph{Pacific J. Math.}, \penalty0 (21):\penalty0 59--78, 1967.

\bibitem[Li et~al.(2023)Li, Yu, and Meng]{li2023scalable}
T.~Li, J.~Yu, and C.~Meng.
\newblock Scalable model-free feature screening via sliced-wasserstein
  dependency.
\newblock \emph{Journal of Computational and Graphical Statistics}, 32\penalty0
  (4):\penalty0 1501--1511, 2023.

\bibitem[Liu et~al.(2022)Liu, Pal, and Harchaoui]{Liu2022}
L.~Liu, S.~Pal, and Z.~Harchaoui.
\newblock Entropy regularized optimal transport independence criterion.
\newblock In G.~Camps-Valls, F.~J.~R. Ruiz, and I.~Valera, editors,
  \emph{Proceedings of The 25th International Conference on Artificial
  Intelligence and Statistics}, volume 151 of \emph{Proceedings of Machine
  Learning Research}, pages 11247--11279. PMLR, 28--30 Mar 2022.

\bibitem[Lyons(2013)]{Lyons2013}
R.~Lyons.
\newblock {Distance covariance in metric spaces}.
\newblock \emph{The Annals of Probability}, 41\penalty0 (5):\penalty0 3284 --
  3305, 2013.

\bibitem[Manole and Niles-Weed(2024)]{ManoleWeed2024}
T.~Manole and J.~Niles-Weed.
\newblock {Sharp convergence rates for empirical optimal transport with smooth
  costs}.
\newblock \emph{The Annals of Applied Probability}, 34\penalty0 (1B):\penalty0
  1108 -- 1135, 2024.

\bibitem[Marti et~al.(2017)Marti, Andler, Nielsen, and Donnat]{Marti2017}
G.~Marti, S.~Andler, F.~Nielsen, and P.~Donnat.
\newblock Exploring and measuring non-linear correlations: Copulas, lightspeed
  transportation and clustering.
\newblock In O.~Anava, A.~Khaleghi, M.~Cuturi, V.~Kuznetsov, and A.~Rakhlin,
  editors, \emph{Proceedings of the Time Series Workshop at NIPS 2016},
  volume~55 of \emph{Proceedings of Machine Learning Research}, pages 59--69,
  Barcelona, Spain, 09 Dec 2017. PMLR.

\bibitem[Mascari et~al.(2026+)Mascari, Lavenant, and
  Catalano]{mascari2026hilbert}
F.~Mascari, H.~Lavenant, and M.~Catalano.
\newblock Hilbert-based correlation indices for measure-valued data.
\newblock \emph{Work in progress}, 2026+.

\bibitem[Mena and Niles-Weed(2019)]{mena2019statistical}
G.~Mena and J.~Niles-Weed.
\newblock {Statistical bounds for entropic optimal transport: sample complexity
  and the central limit theorem}.
\newblock \emph{Advances in neural information processing systems}, 32, 2019.

\bibitem[Merigot and Thibert(2021)]{merigot2021optimal}
Q.~Merigot and B.~Thibert.
\newblock Optimal transport: discretization and algorithms.
\newblock In \emph{Handbook of numerical analysis}, volume~22, pages 133--212.
  Elsevier, 2021.

\bibitem[Mordant and Segers(2022)]{mordant2022measuring}
G.~Mordant and J.~Segers.
\newblock {Measuring dependence between random vectors via optimal transport}.
\newblock \emph{Journal of Multivariate Analysis}, 189:\penalty0 104912, 2022.

\bibitem[M{\'o}ri and Sz{\'e}kely(2019)]{mori2019four}
T.~F. M{\'o}ri and G.~J. Sz{\'e}kely.
\newblock Four simple axioms of dependence measures.
\newblock \emph{Metrika}, 82\penalty0 (1):\penalty0 1--16, 2019.

\bibitem[M{\'o}ri and Sz{\'e}kely(2020)]{mori2020earth}
T.~F. M{\'o}ri and G.~J. Sz{\'e}kely.
\newblock The earth mover's correlation.
\newblock \emph{Annales Universitatis Scientiarum Budapestinensis de Rolando
  Eötvös Nominatae, Sectio Computatorica}, 50, 2020.

\bibitem[Nadjahi et~al.(2020)Nadjahi, Durmus, Chizat, Kolouri, Shahrampour, and
  Simsekli]{nadjahi2020statistical}
K.~Nadjahi, A.~Durmus, L.~Chizat, S.~Kolouri, S.~Shahrampour, and U.~Simsekli.
\newblock {Statistical and topological properties of sliced probability
  divergences}.
\newblock \emph{Advances in Neural Information Processing Systems},
  33:\penalty0 20802--20812, 2020.

\bibitem[Neykov et~al.(2024)Neykov, Wasserman, Kim, and
  Balakrishnan]{Neykov2024}
M.~Neykov, L.~Wasserman, I.~Kim, and S.~Balakrishnan.
\newblock Nearly minimax optimal wasserstein conditional independence testing.
\newblock \emph{Information and Inference: A Journal of the IMA}, 13\penalty0
  (4):\penalty0 iaae033, 12 2024.
\newblock ISSN 2049-8772.

\bibitem[Nies et~al.(2025)Nies, Staudt, and Munk]{nies2025transport}
T.~G. Nies, T.~Staudt, and A.~Munk.
\newblock {Transport dependency: Optimal transport based dependency measures}.
\newblock \emph{The Annals of Applied Probability}, 35\penalty0 (4):\penalty0
  2292--2362, 2025.

\bibitem[Nietert et~al.(2022)Nietert, Goldfeld, Sadhu, and
  Kato]{nietert2022statistical}
S.~Nietert, Z.~Goldfeld, R.~Sadhu, and K.~Kato.
\newblock {Statistical, robustness, and computational guarantees for sliced
  Wasserstein distances}.
\newblock \emph{Advances in Neural Information Processing Systems},
  35:\penalty0 28179--28193, 2022.

\bibitem[Nutz(2021)]{nutz2021introduction}
M.~Nutz.
\newblock Introduction to entropic optimal transport.
\newblock \emph{Lecture notes, Columbia University}, 2021.

\bibitem[Orlin(1997)]{orlin1997polynomial}
J.~B. Orlin.
\newblock A polynomial time primal network simplex algorithm for minimum cost
  flows.
\newblock \emph{Mathematical Programming}, 78:\penalty0 109--129, 1997.

\bibitem[Ozair et~al.(2019)Ozair, Lynch, Bengio, van~den Oord, Levine, and
  Sermanet]{Ozair2019}
S.~Ozair, C.~Lynch, Y.~Bengio, A.~van~den Oord, S.~Levine, and P.~Sermanet.
\newblock Wasserstein dependency measure for representation learning.
\newblock In H.~Wallach, H.~Larochelle, A.~Beygelzimer, F.~d\textquotesingle
  Alch\'{e}-Buc, E.~Fox, and R.~Garnett, editors, \emph{Advances in Neural
  Information Processing Systems}, volume~32. Curran Associates, Inc., 2019.

\bibitem[Panaretos and Zemel(2020)]{panaretos2020invitation}
V.~M. Panaretos and Y.~Zemel.
\newblock \emph{An invitation to statistics in Wasserstein space}.
\newblock Springer Nature, 2020.

\bibitem[Pearson(1920)]{Pearson1920}
K.~Pearson.
\newblock Notes on the history of correlation.
\newblock \emph{Biometrika}, 13\penalty0 (1):\penalty0 25--45, 10 1920.

\bibitem[Petersen and Müller(2019)]{PetersenMueller2019}
A.~Petersen and H.-G. Müller.
\newblock Wasserstein covariance for multiple random densities.
\newblock \emph{Biometrika}, 106\penalty0 (2):\penalty0 339--351, 04 2019.
\newblock ISSN 0006-3444.

\bibitem[Peyr{\'e} and Cuturi(2019)]{Peyre2019computational}
G.~Peyr{\'e} and M.~Cuturi.
\newblock Computational optimal transport -- with applications to data science.
\newblock \emph{Foundations and Trends in Machine Learning}, 11\penalty0
  (5-6):\penalty0 355--607, 2019.

\bibitem[Puccetti(2022)]{Puccetti2022}
G.~Puccetti.
\newblock Measuring linear correlation between random vectors.
\newblock \emph{Information Sciences}, 607:\penalty0 1328--1347, 2022.
\newblock ISSN 0020-0255.

\bibitem[Rabin et~al.(2011)Rabin, Peyr{\'e}, Delon, and
  Bernot]{rabin2011wasserstein}
J.~Rabin, G.~Peyr{\'e}, J.~Delon, and M.~Bernot.
\newblock {Wasserstein barycenter and its application to texture mixing}.
\newblock In \emph{International conference on scale space and variational
  methods in computer vision}, pages 435--446. Springer, 2011.

\bibitem[Santambrogio(2015)]{Santambrogio2015}
F.~Santambrogio.
\newblock Optimal transport for applied mathematicians.
\newblock \emph{Birkh{\"a}user, NY}, 55\penalty0 (58-63):\penalty0 94, 2015.

\bibitem[Sato(1999)]{Sato1999}
K.~Sato.
\newblock \emph{L{\'e}vy Processes and Infinitely Divisible Distributions}.
\newblock Cambridge Studies in Advanced Mathematics. Cambridge University
  Press, Cambridge, 1999.
\newblock ISBN 9780521553025.

\bibitem[Scarsini(1984)]{Scarsini1984}
M.~Scarsini.
\newblock On measures of concordance.
\newblock \emph{Stochastica}, 8\penalty0 (3):\penalty0 201--218, 1984.

\bibitem[Sejdinovic et~al.(2013)Sejdinovic, Sriperumbudur, Gretton, and
  Fukumizu]{Sejdinovic2013}
D.~Sejdinovic, B.~Sriperumbudur, A.~Gretton, and K.~Fukumizu.
\newblock {Equivalence of distance-based and RKHS-based statistics in
  hypothesis testing}.
\newblock \emph{The Annals of Statistics}, 41\penalty0 (5):\penalty0 2263 --
  2291, 2013.

\bibitem[Shannon(1948)]{Shannon1948}
C.~E. Shannon.
\newblock A mathematical theory of communication.
\newblock \emph{The Bell System Technical Journal}, 27\penalty0 (3):\penalty0
  379--423, 1948.

\bibitem[Spearman(1904)]{Spearman1904}
C.~Spearman.
\newblock The proof and measurement of association between two things.
\newblock \emph{The American Journal of Psychology}, 15\penalty0 (1):\penalty0
  72--101, 1904.
\newblock ISSN 00029556.

\bibitem[Székely et~al.(2007)Székely, Rizzo, and Bakirov]{Szekely2007}
G.~J. Székely, M.~L. Rizzo, and N.~K. Bakirov.
\newblock Measuring and testing dependence by correlation of distances.
\newblock \emph{Ann. Statist.}, 35\penalty0 (6):\penalty0 2769--2794, 2007.

\bibitem[Van~der Vaart(2000)]{van2000asymptotic}
A.~W. Van~der Vaart.
\newblock \emph{Asymptotic statistics}, volume~3.
\newblock Cambridge university press, 2000.

\bibitem[Villani(2009)]{villani2009optimal}
C.~Villani.
\newblock \emph{Optimal transport: old and new}, volume 338.
\newblock Springer, Heidelberg, 2009.

\bibitem[Warren(2021)]{warren2021wasserstein}
A.~Warren.
\newblock Wasserstein conditional independence testing.
\newblock \emph{arXiv preprint arXiv:2107.14184}, 2021.

\bibitem[Wiesel(2022)]{wiesel2022measuring}
J.~C. Wiesel.
\newblock {Measuring association with Wasserstein distances}.
\newblock \emph{Bernoulli}, 28\penalty0 (4):\penalty0 2816--2832, 2022.

\bibitem[Xiao and Wang(2019)]{XiaoWang2019}
Y.~Xiao and W.~Y. Wang.
\newblock Disentangled representation learning with {W}asserstein total
  correlation, 2019.

\bibitem[Yitzhaki(2003)]{yitzhaki2003gini}
S.~Yitzhaki.
\newblock {Gini’s mean difference: A superior measure of variability for
  non-normal distributions}.
\newblock \emph{Metron}, 61\penalty0 (2):\penalty0 285--316, 2003.

\bibitem[Zhao et~al.(2021)Zhao, Nie, Wang, and Huang]{zhao2021sliced}
Z.~Zhao, J.~Nie, C.~Wang, and L.~Huang.
\newblock Sliced wasserstein based canonical correlation analysis for
  cross-domain recommendation.
\newblock \emph{Pattern Recogn. Lett.}, 150\penalty0 (C):\penalty0 33–39,
  Oct. 2021.
\newblock ISSN 0167-8655.

\bibitem[Zhou et~al.(2021)Zhou, Lin, and Yao]{zhou2021intrinsic}
H.~Zhou, Z.~Lin, and F.~Yao.
\newblock {Intrinsic Wasserstein correlation analysis}.
\newblock \emph{arXiv preprint arXiv:2105.15000}, 2021.

\end{thebibliography}

\end{document}